\documentclass[preprint]{elsarticle}

\journal{Applied Mathematics and Computation}

\usepackage{amsfonts,amsmath,amssymb,amsthm,mathrsfs,graphicx}
\usepackage{soul}
\usepackage{bm}
\usepackage{caption,subcaption}
\usepackage{paralist}
\usepackage[colorlinks=true,
linkcolor = blue,
urlcolor  = blue,
citecolor = blue,
]{hyperref} %implies package color

\newtheorem{theorem}{Theorem}[section]
\newtheorem{lemma}[theorem]{Lemma}
\newtheorem{proposition}[theorem]{Proposition}
\newtheorem{corollary}[theorem]{Corollary}
\newtheorem{conjecture}{Conjecture}[section]
\newtheorem{definition}{Definition}[section]
\newtheorem{remark}{\textbf{Remarks}}
\newtheorem{example}{Example}
\newcommand{\E}{{\rm I}\kern-0.18em{\rm E}}

\newcommand{\Q}{\ensuremath{\mathbb{Q}}}
\newcommand{\R}{\ensuremath{\mathbb{R}}}
\newcommand{\C}{\ensuremath{\mathbb{C}}}
\newcommand{\ud}{\ensuremath{\mathrm{d}}}
\newcommand{\uinner}[2]{\ensuremath{\langle #1 ,\; #2 \rangle}}

\newcommand{\wa}{\ensuremath{w_{0}^{*}}}
\newcommand{\wb}{\ensuremath{w^{*}}}
\newcommand{\tr}{\ensuremath{\mathrm{tr}}}

\usepackage{amsopn}

\makeatletter
\renewcommand*\env@matrix[1][*\c@MaxMatrixCols c]{%
  \hskip -\arraycolsep
  \let\@ifnextchar\new@ifnextchar
  \array{#1}}

\begin{document}

\begin{frontmatter}

\title{Identifiability Analysis of Linear Ordinary Differential Equation Systems with a Single Trajectory\tnoteref{mytitlenote}}
\tnotetext[mytitlenote]{This work was funded in part by the National Institutes of Health under award number R01 AI087135 and the University of Rochester CTSA award number UL1 TR002001 from the National Center for Advancing Translational Sciences of the National Institutes of Health.}

\author[mymainaddress]{Xing Qiu}\ead{xing_qiu@urmc.rochester.edu}
\author[mysecondaryaddress]{Tao Xu}\ead{tao.xu@uth.tmc.edu}
\author[mysecondaryaddress]{Babak Soltanalizadeh}\ead{babak.soltanalizadeh@uth.tmc.edu}
\author[mysecondaryaddress]{Hulin Wu\corref{mycorrespondingauthor}}
\cortext[mycorrespondingauthor]{Corresponding author}
\ead{hulin.wu@uth.tmc.edu}

\address[mymainaddress]{Department of Biostatistics and Computational Biology \\
  University of Rochester, Rochester, NY, U.S.A.}
\address[mysecondaryaddress]{Department of Biostatistics \& Data Science, School of Public Health \\
  University of Texas Health Science Center at Houston, Houston, TX, U.S.A.}

\begin{abstract}
  Ordinary differential equations (ODEs) are widely used to model dynamical behavior of systems. It is important to perform identifiability analysis prior to estimating unknown parameters in ODEs (a.k.a. inverse problem), because if a system is unidentifiable, the estimation procedure may fail or produce erroneous and misleading results.

  Although several qualitative identifiability measures have been proposed, much less effort has been given to developing \emph{quantitative} (continuous) scores that are robust to uncertainties in the data, especially for those cases in which the data are presented as a single trajectory beginning with one initial value. 

  In this paper, we first derived a closed-form representation of linear ODE systems that are not identifiable based on a single trajectory. This representation helps researchers design practical systems and choose the right prior structural information in practice. Next, we proposed several quantitative scores for identifiability analysis in practice. In simulation studies, the proposed measures outperformed the main competing method significantly, especially when noise was presented in the data. We also discussed the asymptotic properties of practical identifiability for high-dimensional ODE systems and conclude that, without additional prior information, many random ODE systems are practically unidentifiable when the dimension approaches infinity.
\end{abstract}


\end{frontmatter}

% \linenumbers

% Main text entry area
\section{Background and Introduction}

% Xing: the introduction was completely re-written to emphasize more on Stanhope2014, because its methodology is our primary competitor.

Ordinary differential equations (ODE) can be used to model complex dynamic systems in a wide variety of disciplines including economics, physics, engineering, chemistry, and biology~\cite{butcher2014ordinary,commenges2011inference,DeJong2002,hemker1972numerical,holter2001dynamic,huang2006hierarchical,lavielle2011maximum,li2011large,Lu2011,Ramsay2007Parameter}. Such systems are usually represented as
\begin{gather}
  \label{eq:general-state-eqn}
  \begin{cases}
    D\mathbf{x}(t) = f(\mathbf{x}(t), \mathbf{u}(t), \theta),& t\in (0, T), \\
    \mathbf{x}(0) = \mathbf{x}_{0}.
  \end{cases} \\
  \label{eq:general-observation-eqn}
  \mathbf{y}(t) = h(\mathbf{x}(t), \mathbf{u}(t), \theta).
\end{gather}
Here $\mathbf{x}(t) = (x_{1}(t), \dots, x_{d}(t))'\in \R^{d}$ is the state vector, $D = \frac{\ud }{\ud t}$ is the first order differential operator\footnote{To avoid confusion, we reserve symbol $'$ (apostrophe) for matrix transpose, not the derivative with respect to $t$.},  $\mathbf{y}(t) \in R^{d}$ is the output vector, $\mathbf{u}(t)$ is a known system input vector, $f$, $g$ are known families of linear or nonlinear functions indexed by $\theta \in \R^{p}$, which is the vector of unknown parameters to be estimated.  Equation \eqref{eq:general-state-eqn} is called the state equation and Equation~\eqref{eq:general-observation-eqn} the output or observation equation. In this paper, we focus on an important special case of the above general ODE system: homogeneous linear ODE system with complete observation
\begin{equation}
  \label{eq:lin-ode}
  D\mathbf{x}(t) = A \mathbf{x}(t), \qquad \mathbf{x}(0) = \mathbf{x}_{0} = (x_{1,0}, \dots, x_{d,0})', \qquad \mathbf{y}(t) = \mathbf{x}(t).
\end{equation}
Here $A \in M_{d\times d}$ is a matrix (called the system matrix) that characterizes the mechanistic relationship between $x_{i}(t)$; $\mathbf{y}(t)=\mathbf{x}(t)$ means that we can directly observe $\mathbf{x}(t)$ in all dimensions.

Let $\mathbf{x}(t|A, \mathbf{x}_{0})$ be the solution curve (a.k.a. the trajectories) of Equation~\eqref{eq:lin-ode} initiated at $\mathbf{x}_{0}$ and governed by system matrix $A$. It is well known that $\mathbf{x}(t|A, \mathbf{x}_{0})$ can be represented as a unique matrix exponential
\begin{equation}
  \label{eq:matrix-exponentials}
  \mathbf{x}(t|A, \mathbf{x}_{0}) = e^{tA} \mathbf{x}_{0}.
\end{equation}

As such, the \emph{forward problem} of Equation~\eqref{eq:lin-ode}, defined as solving the ODE system with given $A$ and $\mathbf{x}_{0}$, has been resolved in the mathematical sense -- despite of several known numerical issues in matrix exponentials for high-dimensional data~\cite{moler2003nineteen}.

In practical applications, the parameters that characterize the ODE system, such as $A$ and $\mathbf{x}_{0}$ in Equation~\eqref{eq:lin-ode}, must be estimated from the real data. This is known as the \emph{inverse problem}. Over the years, many parameter estimation methods have been developed for ODE systems~\cite{commenges2011inference,huang2006hierarchical,huang2006bayesian,huang2010hierarchical,lavielle2011maximum,li2005parameter,putter2002bayesian,wu2019parameter,Xue2010}.

In principle, before performing the parameter estimation, we need to address an important question: are the parameters in a particular ODE model \emph{identifiable} from the data? In this context, ``identifiability'' loosely means that there is a \emph{unique} mapping between the trajectories and the parameters of a family of ODE systems.

By now, a rich literature on the identifiability of both linear and nonlinear ODE systems is available, see~\cite{miao2011identifiability} for a thorough review of these methods. Unfortunately, most of them only consider the identifiability of the system matrix ($A$), and assume that one can choose an arbitrary initial condition $\mathbf{x}_{0} \in \R^{d}$. For example, the \emph{global identifiability} used in some literature on nonlinear ODE identifiability (e.g., \cite{thowsen1978identifiability}) reduces to the following definition for Equation~\eqref{eq:lin-ode}, as pointed out in~\cite{stanhope2014identifiability}:
\begin{definition}
  Linear ODE system \eqref{eq:lin-ode} is globally identifiable in a subset $\Omega \subset M_{d\times d}$ iff for all $A,B \in \Omega$, $A\ne B$, there exists $\mathbf{x}_{0} \in \R^{d}$, such that $\mathbf{x}(t|A, \mathbf{x}_{0}) \ne \mathbf{x}(t|B, \mathbf{x}_{0})$.
\end{definition}

However, such definition is of little use for linear ODE systems because Stanhope and colleagues proved in \cite{stanhope2014identifiability} that, due to the linearity of Model~\eqref{eq:lin-ode}, such ODE system is \emph{always globally identifiable} in the entire parameter space $M_{d\times d}$. Consequently, no computationally intensive symbolic computation on global identifiability is needed for linear ODE systems. The above definition of identifiability is also impractical because in many real world applications, data are only available in the form of \textbf{one trajectory} starting with a single $\mathbf{x}_{0}$. For example, influenza infection affects the state of transcriptome of a patient, which can be modeled by Equation~\eqref{eq:lin-ode} (\cite{qiu2015diversity,sun2016controllability,wu2013high,wu2014modeling}). However, it is currently impossible for a researcher to select an arbitrary $\mathbf{x}_{0}$ even in an animal study, because not only we do not have the technology to  alter whole transcriptome globally, but also not all transcriptome states are biologically feasible. Furthermore, we cannot repeat the same $\mathbf{x}_{0}$ for a subject either, because the infection can have long-lasting effects to the immune system of that subject~\cite{mccullers2010influenza}.

As a response to this weakness, several researchers developed a concept known as locally strong identifiability~\cite{tunali1987new} or $\mathbf{x}_{0}$-identifiability~\cite{jeffrey2005identifiability}, that involves data with only one trajectory. For Equation~\eqref{eq:lin-ode}, it can be stated as follows.
\begin{definition}\label{def:x0-ident}
  Equation~\eqref{eq:lin-ode} is
  $\mathbf{x}_{0}$-identifiable w.r.t. a given $\mathbf{x}_{0}$ iff
  there exists an open and dense subset
  $\Omega \subset M_{d\times d}$, such that for all
  $A, B \in \Omega$, $A\ne B$, we have
  $\mathbf{x}(t|A, \mathbf{x}_{0}) \ne \mathbf{x}(t|B,
  \mathbf{x}_{0})$ on $(0,\delta t)$, for some $0 < \delta t < T$.
\end{definition}

Of note, the following natural extension to
$\mathbf{x}_{0}$-identifiability was proposed in~\cite{jeffrey2005identifiability}:
\begin{definition}\label{def:struct-ident}
  System ~\eqref{eq:lin-ode} is structurally identifiable iff there
  exist open and dense subsets $\Omega \subset M_{d\times d}$,
  $M^{0}\in \R^{d}$, such that for all $A,B \in \Omega$, $A\ne B$
  and all $\mathbf{x}_{0} \in M^{0}$, we have
  $\mathbf{x}(t|A, \mathbf{x}_{0}) \ne \mathbf{x}(t|B,
  \mathbf{x}_{0})$ on $(0,\delta t)$, for some $0 < \delta t < T$.
\end{definition}

In other words, Definition \ref{def:struct-ident} is $\mathbf{x}_{0}$-identifiability that applies to not one $\mathbf{x}_{0}$, but an open and dense set $M^{0}\in \R^{d}$. This definition is consistent with the
\emph{structural identifiability}~\cite{xia2003identifiability} and \emph{geometrical identifiability}~\cite{tunali1987new} for nonlinear ODE systems.

As a remark, the \emph{open and dense} condition of $\Omega$ was designed to rule out a set of certain ``inconvenient'' parameters that has zero-measure. For example, it can be shown that if $A$ has repeated eigenvalues, it is not identifiable with a class of other system matrices (see Section~\ref{sec:repeated-eigenvalues} and example~\ref{example:I2} in Supplementary Text for more details). A workaround is to simply define $\Omega$ to be those matrices with no repeated eigenvalues, which is clearly a dense open set in $M_{d\times d}$.  However, it is theoretically possible that the said open and dense set $\Omega$ may be ``small'' compared with $M_{d\times d}$ in terms of a measure such as $\lambda_{d\times d}$, the Lebesgue measure. See Example \ref{exp:open-dense-small} in Supplementary Text, Section~\ref{sec:additional-examples} for such an example. This can be seen as a weakness because in most real world applications, there is uncertainty in $A$, so we want to ensure that the identifiability applies to \emph{almost every} $A \in M_{d\times d}$, not just a dense set with small measure or probability. As a concrete example, $A$ may be modeled as $M + E$, where $M$ is a deterministic matrix and $E$ a perturbation term sampled from a random matrix distribution such as the real Ginibre ensemble~(GinOE, \cite{ginibre1965statistical}). By definition, if $E \sim \mathrm{GinOE}$, $A_{ij}$ are $i.i.d.$ standard normal random variables, therefore the probability measure associated with GinOE and the Lebesgue measure are absolutely continuous with each other. Therefore, the condition that almost every $A$ is identifiable is equivalent to requiring $\lambda_{d\times d}\left( \Omega^{c} \right) = 0$, where $\Omega^{c} := M_{d\times d}\setminus \Omega$ is the complement of $\Omega$, the collection of all identifiable $A$.

To the best of our knowledge, the most systematic study of linear ODE identifiability from a single observed trajectory is provided in~\cite{stanhope2014identifiability}. In this seminal work, Stanhope and colleagues derived several necessary and sufficient conditions of identifiability that applies to a single trajectory. Specifically, they proposed two concepts, called ``identifiability for a single trajectory'' and ``unconditional identifiability'', defined as follows.

\begin{definition}[identifiability for a single trajectory]
  \label{def:as-x0-ident}
  System~\eqref{eq:lin-ode} is identifiable for a single trajectory in
  $\Omega \subset M_{d\times d}$ and a given $\mathbf{x}_{0}$ iff for
  all $A, B \in \Omega$, $A\ne B$, we have
  $\mathbf{x}(t|A, \mathbf{x}_{0}) \ne \mathbf{x}(t|B,
  \mathbf{x}_{0})$, for some $0 < t < T$.
\end{definition}

\begin{definition}[unconditional identifiability]
  \label{def:uncond-ident}
  System~\eqref{eq:lin-ode} is unconditionally identifiable in
  $\Omega \in M_{d\times d}$ iff for all $A, B \in \Omega$, $A\ne B$ implies that for each nonzero $\mathbf{x}_{0} \in \R^{d}$, $\mathbf{x}(t|A, \mathbf{x}_{0}) \ne \mathbf{x}(t|B,
  \mathbf{x}_{0})$, for some $0 < t < T$.
\end{definition}

Between these two definitions, Definition~\ref{def:uncond-ident} adheres more to the traditional definition of structural identifiability for nonlinear ODE systems. Roughly speaking, it means that System~\eqref{eq:lin-ode} is identifiable from a single trajectory initiated from \emph{every} $\mathbf{x}_{0} \in \R^{d}$. Unfortunately, it is of little practical use for linear ODE system because \textbf{no system} satisfies this condition for an unconstrained parameter estimation problem, namely, $\Omega = M_{d\times d}$. In fact, we showed that (Supplementary Text, Section~\ref{sec:about-struct-ident}): (a) when the dimension $d$ is odd, unconditional identifiability is not attainable for all $\Omega \subseteq M_{d\times d}$, and (b) when $d$ is even, unconditional identifiability is not attainable for all $\Omega \subseteq M_{d\times d}$ such that $\lambda_{d\times d}\left(M_{d\times d}\setminus \Omega\right) = 0$.  In summary, a large body of prior work in identifiability analysis are geared towards nonlinear ODEs with arbitrarily many observed trajectories, which is of little utility to linear ODEs, and this issue cannot be fixed by simply removing a zero-measure set from their definitions.

One major contribution of Stanhope and colleagues is that they established a beautiful connection between the algebraic and geometric aspects of linear ODE systems in \cite[Theorem (3.4)]{stanhope2014identifiability}. We find it easier to state this important result by first define the following minimalist definition of identifiability.

\begin{definition}[$(A,\mathbf{x}_{0})$-identifiability]
  \label{def:mini-ident}
  For system~\eqref{eq:lin-ode}, we call $A$ is identifiable at
  $\mathbf{x}_{0}$ if for all $B\in M_{d\times d}$,
  $\mathbf{x}(t|A, \mathbf{x}_{0}) \ne \mathbf{x}(t|B,
  \mathbf{x}_{0})$, for some $0 < t < T$.
\end{definition}
\begin{remark}
  Definition~\ref{def:mini-ident} is not equivalent to Definition~\ref{def:as-x0-ident} applied to $\Omega := M_{d\times d}$ because in Definition~\ref{def:mini-ident}, $A$ is fixed and $B$ is an arbitrary matrix in $\Omega$, while in Definition~\ref{def:as-x0-ident}, both $A$ and $B$ are arbitrary matrices in $\Omega$.  In short, Definition~\ref{def:mini-ident} is an intrinsic property of a \emph{single} system, not a \emph{collective} property of a set of system matrices.
\end{remark}

Using this definition, \cite[Theorem (3.4)]{stanhope2014identifiability} can be restated as follows: the $(A,\mathbf{x}_{0})$-identifiability holds if and only if the solution curve $\mathbf{x}(t|A)$ is not contained in a \emph{proper invariant subspace} of $A$. Based on this powerful theoretical result, they proposed to use $\kappa(X_{1})$, the condition number of the matrix of a subset of discrete observations (see Section~\ref{sec:stanhope-cond-num} for more details), to test the identifiability for discrete data with noise in practice.

However, their study is not without shortcomings. First, they did not derive the explicit structure of the largest subset $\Omega \subseteq M_{d\times d}$ for a give $\mathbf{x}_{0}$ in identifiability analysis for a single trajectory, nor the equivalent class of all $B\in M_{d\times d}$ such that $\mathbf{x}(t|B,t) = \mathbf{x}(t|A,t)$ when the system $A$ is deemed unidentifiable at a given $\mathbf{x}_{0}$.  Secondly, while using $\kappa(X_{1})$ to check the practical identifiability of an ODE system is a clever heuristic, it has much room for improvement because: a) not all data are used in $\kappa(X_{1})$, therefore it does not utilize data efficiently; b) measurement errors are not directly reflected in this score and there is no analysis of the asymptotic properties of $\kappa(X_{1})$ from the statistical perspective; and c) by definition, $\kappa(X_{1})$ depends on the availability of data at multiple time points, so it requires solving the ODE numerically in simulation studies, which can be time consuming for high-dimensional systems and/or when a large set of systems are considered.

In this study, we first derive a closed-form representation of $(A,\mathbf{x}_{0})$-unidentifiable class, which is defined in Definition~\ref{def:unidentifiably-class} as the collection of system matrices that are not identifiable for a given pair of $A$ and $\mathbf{x}_{0}$. We also provide explicit structures of the equivalent class of unidentifiable systems due to repeated eigenvalues in $A$ in Supplementary Text, Section~\ref{sec:repeated-eigenvalues}.  We believe these results will be valuable for future studies that combine \textit{a priori} topological constraints (e.g., knowing which entries in $A$ are zero in advance) and identifiability. In light this, we give a brief discussion of the best practice of using  prior information to resolve the identifiability issues in Supplementary Text, Section~\ref{sec:prior-inform-ident}. More systematic studies in this direction warrant a future study.

Secondly, we specify explicit, computable principles of $(A,\mathbf{x}_{0})$-identifiability based on either $\mathbf{x}_{0}$ or the entire solution trajectory.  These results are presented in our Theorems~\ref{thm:structure} and \ref{thm:identifiability-based-on-pairwise-inprod}. To assist practical identifiability analyses, we propose three continuous scores: the initial condition-based identifiability score (ICIS, denoted as $\wa$ in Equation~\eqref{eq:w0star-def}), the smoothed condition number (SCN, denoted as $\tau$ in Equation~\eqref{eq:tau-def}), and the practical identifiability score (PIS, denoted as $\wb$ in Equation~\eqref{eq:wstar-def}), to solve the aforementioned problems. ICIS only uses $A$ and $\mathbf{x}_{0}$, therefore it does not require numerically solving the ODE before the identifiability analysis. We think ICIS is most suitable for designing simulated ODE systems independent of a specific set of real data. SCN and PIS use data from all time points, which are more suitable for practical identifiability analysis with real data.  Using extensive simulation studies, we showed that SCN and PIS correlated with practical identifiability significantly better than $\kappa(X_{1})$ when there was noise in the data.

In addition, we studied the asymptotic properties of practical identifiability for high-dimensional systems with randomly generated $A$ and $\mathbf{x}_{0}$. We reached the following interesting conclusions: a) almost every system is $(A,\mathbf{x}_{0})$-identifiable in the sense that $\mathrm{ICIS}>0$; and b) when $d\to \infty$, almost all systems are practically unidentifiable in the sense that $\mathrm{ICIS} \to 0$. These two seemly contradictory conclusions suggest that the practical identifiability of high-dimensional ODE systems is very different from that of low-dimensional systems, and classical mathematical identifiability analyses are insufficient for analyzing \emph{high-dimensional} real world applications. The focus must be shifted towards practical identifiability analyses characterized by \emph{continuous} scores, especially with the considerations from the stochastic perspective.

Last but not the least, we provide a user-friendly R package \texttt{ode.ident}, with full
documentation and examples, so practitioners with minimum programming skills can analyze the identifiability of linear ODE systems. This R package is available at \url{https://github.com/qiuxing/ode.ident}.

\section{$(A,\mathbf{x}_{0})$-identifiability}
\label{sec:A-x0-identifiability}

In this section, we focus on the mathematical inverse problem for one fully observed trajectory. Namely, we assume that we have the complete observation of \emph{one} solution curve $\mathbf{x}(t)= (x_{1}(t), \dots, x_{d}(t))' \in \mathbb{R}^{d}$) governed by Equation~\eqref{eq:lin-ode} and its derivative on $[0,T]$, with no measurement error.

First, let us define the $(A,\mathbf{x}_{0})$-unidentifiable class as follows.

\begin{definition}[$(A,\mathbf{x}_{0})$-unidentifiable class]
  \label{def:unidentifiably-class}
  For a given system matrix $A\in M_{d\times d}$ and initial condition $\mathbf{x}_{0} \in \R^{d}$, the $(A,\mathbf{x}_{0})$-unidentifiable class, denoted by $[A]_{\mathbf{x}_{0}}$, is a subset of matrices in $M_{d\times d}$ such that
  \begin{equation}
    \label{eq:A-x0-class-def}
    B \in [A]_{\mathbf{x}_{0}} \quad \text{iff} \quad \mathbf{x}(t|A, \mathbf{x}_{0}) = \mathbf{x}(t|B, \mathbf{x}_{0}).
  \end{equation}
  In other words, two system matrices $A,B$ are in the same unidentifiable class if and only if they produce the same solution trajectory at $\mathbf{x}_{0}$.
\end{definition}

The overarching goal of this section is to understand the structure of $[A]_{\mathbf{x}_{0}}$, and the conditions under which this class contains only \textbf{one} member, therefore $A$ can be uniquely determined by the trajectory $\mathbf{x}(t|A, \mathbf{x}_{0})$. To this end, we need to introduce an important geometric concept called \textbf{invariant subspace}, which is a generalization of eigenvectors, and its connection to the Jordan decomposition of $A$ in Section~\ref{sec:jordan-decomp-invariant-subspaces}\footnote{These concepts and results can be found in many graduate level matrix analysis textbooks, e.g., \cite{gohberg2006invariant}.}.

\subsection{Jordan Decomposition and Invariant subspaces}
\label{sec:jordan-decomp-invariant-subspaces}

\begin{definition}[Invariant subspace]
  An invariant subspace of a square matrix $A_{d\times d}$ is a linear subspace $L \subseteq R^{d}$ such that for all $\mathbf{x}\in L$, $A\mathbf{x} \in L$.  We say $L$ is a \emph{proper} invariant subspace if $L \neq \R^{d}$.
\end{definition}

By definition, we see that if a vector $\mathbf{x}$ is in a proper invariant subspace $L$ of $A$, $A\mathbf{x}$ must also stay in $L$. Using mathematical induction, we see that $A^{n}\mathbf{x} \in L$ for every positive integer $n$. With a little more work, it can be proven that $e^{tA}\mathbf{x} \in L$ for $t\in [0,T]$, where $e^{tA}$ is the matrix exponential of $tA$.

The following proposition states that the intersection and linear span (the combination) of two invariant subspaces are invariant subspaces.
\begin{proposition}
  If $L_{1}$ and $L_{2}$ are invariant subspaces of $A$, then
  \begin{enumerate}
  \item $L_{1} \cap L_{2}$ is an invariant subspace of $A$;
  \item $\mathrm{span}(L_{1}, L_{2})$ is an invariant subspace of $A$.
  \end{enumerate}
  In other words, the collection of invariant subspaces of $A$ forms a \emph{lattice}.
\end{proposition}

Based on random matrix theory~\cite{ginibre1965statistical,lehmann1991eigenvalue,tao2012topics}, we know that almost every (w.r.t. the Lebesgue measure on $M_{d\times d}$) $A \in M_{d\times d}$ has $d$ \textbf{distinct eigenvalues}. This conclusion also holds for probability measures associated with most random matrix ensembles such as Ginibre ensemble, Gaussian orthogonal ensemble, Wishart ensemble, etc.\cite{tao2012topics}

Consequently, almost every $A \in M_{d\times d}$ has the following Jordan decomposition
\begin{equation}
  \label{eq:jordan-decomp}
  \begin{gathered}
    A = Q \Lambda Q^{-1}, \qquad \Lambda =
    \begin{pmatrix}
      J_{1} & &  \\
      & \ddots & \\
      & & J_{K}
    \end{pmatrix}, \qquad Q =
    \begin{pmatrix}[c|c|c|c]
      Q_{1} & Q_{2} & \dots & Q_{K}
    \end{pmatrix}.
    \\
    J_{k} =
    \begin{cases}
      c_{k}, & k=1, 2, \dots, K_{1}, \\
      \left(
        \begin{smallmatrix}
          a_{k} & -b_{k} \\
          b_{k} & a_{k}
        \end{smallmatrix}\right), & k=K_{1}+1, K_{1}+2, \dots, K.
    \end{cases} \\
    \dim Q_{k} =
    \begin{cases}
      1, & k= 1,\dots, K_{1}, \\
      2, & k=K_{1}+1, \dots, K.
    \end{cases}
  \end{gathered}
\end{equation}

In other words, $A$ can be decomposed into $K=K_{1}+K_{2}$ Jordan blocks, the first $K_{1}$ such blocks are $1\times 1$ blocks corresponding with real eigenvalues (those $c_{k}$ in Equation~\eqref{eq:jordan-decomp}); and the rest $K_{2}$ blocks are $2\times 2$ blocks corresponding with complex eigenvalues $a_{k} \pm b_{k}i$.  There is a corresponding column-wise decomposition of matrix $Q$, such that each $Q_{k}$ contains: (a) a single column vector of $Q$ which is the eigenvector of $c_{k}$, or (b) two column vectors in $Q$ such that $Q_{k} := (\mathbf{v}_{k1} | \mathbf{v}_{k2})$, which are the ``eigenvectors'' associated with $a_{k} \pm b_{k}i$.

Note that the word ``eigenvector'' in case (b) refers to a generalization of true eigenvectors. In fact, those $2\times 2$ Jordan blocks do not have real eigenvectors; instead, each of them is associated with a 2-dimensional invariant subspace of $A$ and $Q_{k} = (\mathbf{v}_{k1} | \mathbf{v}_{k2})$ is a basis of this 2-dimensional invariant subspace.

We would like to point out that based on simple enumeration of dimensions, we have $d = K_{1}+2K_{2}$, and
\begin{equation}
  \label{eq:dimensions}
  Q_{k} =
  \begin{cases}
    Q_{\cdot k}, & k=1,2,\dots, K_{1}, \\
    (\mathbf{v}_{k1} | \mathbf{v}_{k2}), \quad \mathbf{v}_{k1} = Q_{\cdot 2k-K_{1}-1},\; \mathbf{v}_{k2} = Q_{\cdot 2k-K_{1}}, & k=K_{1}+1, \dots, K.
  \end{cases}
\end{equation}

In other words, $\mathbf{v}_{k1}$ and $\mathbf{v}_{k2}$ in $Q_{k}$ are the ($2k-K_{1}-1$)-th and ($2k-K_{1}$)-th column vectors of $Q$, respectively. For convenience, we define the following correspondences between $i$ (the original dimension in $J$) and $k$ (the index of invariant subspaces):
\begin{equation}
  \label{eq:i-k-functions}
  i(k):= 2k-K_{1}-1, \qquad k(i) :=
  \begin{cases}
    i, & i=1, 2, \dots, K_{1}, \\
    K_{1} + \lceil \frac{i-K_{1}}{2} \rceil, & i=K_{1}+1, \dots, K.
  \end{cases}
\end{equation}
Using the above notation, $\mathbf{v}_{k1} = Q_{\cdot i(k)},\; \mathbf{v}_{k2} = Q_{\cdot i(k)+1}$.

\begin{theorem}
  Let $L_{k} := \mathrm{span}(Q_{k})$. Each $L_{k}$ is an invariant subspace of $A$.  Furthermore, if $L$ is an invariant subspace of $A$, it can always be decomposed as
  \begin{equation}
    \label{eq:invariant-subspace-decomp}
    L = \mathrm{span} \bigcup_{i\in S} L_{k}, \qquad S \subseteq \left\{ 1, 2, \dots, K \right\}.
  \end{equation}
\end{theorem}

From now on, we assume that $A$ has $d$ distinct eigenvalues and can be decomposed as $A = Q \Lambda Q^{-1}$ in Equation~\eqref{eq:jordan-decomp}. The case in which $A$ has repeated eigenvalues will be discussed in Supplementary Text, Section~\ref{sec:repeated-eigenvalues}. For convenience, we will also denote $L_{0} := \left\{ 0_{d} \right\}$, the trivial proper invariant subspace of $A$ that contains only the origin.

\subsection{Initial Condition-based Identifiability Score (ICIS)}
\label{sec:w0star}

One of our main conclusion is that the $(A,\mathbf{x}_{0})$-identifiability defined in Definition~\ref{def:mini-ident} can be determined by the initial condition-based identifiability score (ICIS) defined as follows.

\begin{definition}\label{def:ICIS}
  Let $\mathbf{x}_{0} \in \R^{d}$ and $\tilde{\mathbf{x}}_{0} := Q^{-1}\mathbf{x}_{0} \in \R^{d}$. We define the $\wa$ statistic in the following equation as the Initial Condition-based Identifiability Score (ICIS):
  \begin{equation}\label{eq:w0star-def}
    \begin{gathered}
      w_{0,k} :=
      \begin{cases}
        \tilde{\mathbf{x}}_{0,k} \in \R^{1}, & k=1,2,\dots, K_{1}, \\
        (\tilde{\mathbf{x}}_{0, i(k)}, \tilde{\mathbf{x}}_{0, i(k)+1})' \in \R^{2}, & k=K_{1}+1, \dots, K.
      \end{cases}. \\
      \wa := \min_{k} |w_{0,k}|.
    \end{gathered}
  \end{equation}
  Here $|w_{0,k}|$ is the absolute value of $w_{0,k}$ for $k=1,2,\dots, K_{1}$, and the Euclidean norm of $w_{0,k}$ for $k=K_{1}+1, K_{1}+2, \dots, K$.
\end{definition}

From the geometric perspective, $\mathbf{x}_{0}$ can be decomposed into a linear combination (oblique projections) of $Q_{k}$, and $w_{0,k}$ are the linear coefficients of such a decomposition
\begin{equation}\label{eq:linear-decomp-x0}
  \mathbf{x}_{0} = Q \tilde{\mathbf{x}}_{0} = \sum_{k=1}^{K_{1}} w_{0,k} Q_{k} .
\end{equation}

Heuristically speaking, if $w_{0,k}= 0$ (in $\R^{1}$ or $\R^{2}$), $\mathbf{x}_{0}$ does not contain any information from $L_{k}$. This is because in this case, $\mathbf{x}_{0} \in L_{-k}$, where $L_{-k}$ is the invariant subspace of $A$ that \emph{excludes} $L_{k}$, which implies that the \emph{entire trajectory}, $\mathbf{x}(t|A, \mathbf{x}_{0})$, is in $L_{-k}$ (see the discussion in the beginning of Section~\ref{sec:jordan-decomp-invariant-subspaces}). These ideas are summarized in Lemma~\ref{lemma:equiv-invariant-subspace} below.

\begin{lemma}\label{lemma:equiv-invariant-subspace}
  The following two statements are equivalent
  \begin{enumerate}
  \item There exists a proper invariant subspace $L \subsetneq \R^{d}$ of $A$, such that $\mathbf{x}_{0} \in L$.
  \item There exists $k \in \left\{ 1,2,\dots, K \right\}$, such that $|w_{0,k}| = 0$, or equivalently, $\wa = 0$.
  \end{enumerate}
\end{lemma}

Now we are ready to present the following theorem:
\begin{theorem}[Computational criterion for $(A,\mathbf{x}_{0})$-identifiability]\label{thm:identifiability}
  Assuming that an ODE system $A$ has $d$ distinct eigenvalues. This system is identifiable at $\mathbf{x}_{0}$ if and only if the ICIS is nonzero.
\end{theorem}
\begin{proof}
  Based on Lemma 3.2 and Theorem (3.4) in~\cite{stanhope2014identifiability},  we know that system $A$ is identifiable at $\mathbf{x}_{0}$ if and only if $\mathbf{x}_{0}$ is not contained in a proper invariant subspace of $A$, which is equivalent to $\wa \ne 0$ based on our Lemma~\ref{lemma:equiv-invariant-subspace}.
\end{proof}

Theorem~\ref{thm:identifiability} implies that, when $A$ is not identifiable at $\mathbf{x}_{0}$, there must be a nonempty subset $S_{0} \subset \left\{ 1, 2, \dots, K \right\}$ such that $|w_{0,k}| = 0$, for $k \in S_{0}$. WLOG, we assume that $\mathbf{x}_{0} \ne 0_{d}$, so its complement set $S_{+} :=  \left\{ 1, 2, \dots, K \right\} \setminus S_{0}$ must be nonempty. By construction, $L_{+} := \mathrm{span} \bigcup_{k\in S_{+}} L_{k}$ is a proper invariant subspace, and $\mathbf{x}_{0} \in L_{+}$.

The two index sets $S_{0}$ and $S_{+}$ induce the following diagonal binary matrices
\begin{equation}
  \label{eq:dummy-matrices-S0-S+}
  I_{0} := \mathrm{diag}(a_{i}), \quad a_{i} =
  \begin{cases}
    1, & k(i)\in S_{0}, \\
    0, & k(i) \ne S_{0}.
  \end{cases} \qquad I_{+} := I_{d\times d} - I_{0}.
\end{equation}

They can be used to construct the following decomposition of the eigenvector matrix $Q$ and the Jordan block matrix $J$
\begin{equation}\label{eq:Qplus-Q0-decomp}
  \begin{gathered}
    Q_{0} := Q I_{0}, \qquad Q_{+} := Q I_{+}, \qquad Q = Q_{0} + Q_{+}. \\
    J_{0} := I_{0} J I_{0} = JI_{0}, \qquad J_{+} := I_{+} J I_{+} = JI_{+}, \qquad J = J_{0} + J_{+}.
  \end{gathered}
\end{equation}
Intuitively, $Q_{0}$ and $J_{0}$ replace column vectors in $Q$ and blocks in $J$ into zeros if they belong to $L_{+}$. $Q_{+}$ and $J_{+}$ are defined in exactly the opposite way.

Using these notation, we describe the explicit structure of $(A,\mathbf{x}_{0})$-unidentifiable class in the following Theorem.

\begin{theorem}[Structure of the unidentifiable classes]\label{thm:structure}
  $(A,\mathbf{x}_{0})$-unidentifiable class has the following explicit structure
  \begin{equation}
    \label{eq:structure-of-equiv-class}
    \begin{split}
      [A]_{\mathbf{x}_{0}} &= Q \left( J_{+} + I_{0} (J_{0} + D) I_{0} \right)Q^{-1} = A + Q (I_{0} D I_{0}) Q^{-1}, \qquad D \in M_{d\times d}.
    \end{split}
  \end{equation}
  In other words, two matrices $A_{1}$ and $A_{2}$ are in the same $(A,\mathbf{x}_{0})$-unidentifiable class (written as $A_{1} \sim A_{2}$) iff there exist $D\in M_{d\times d}$, such that $A_{1} - A_{2} = Q (I_{0}D I_{0})Q^{-1}$.
\end{theorem}
\begin{proof}
  See Section~\ref{sec:mathematical-proofs}, Supplementary Text.
\end{proof}

\begin{remark}
  The degrees of freedom in $[A]_{\mathbf{x}_{0}}$ is controlled by $D_{0} := I_{0}DI_{0}$, which has $d_{0}^{2}$ degrees of freedom (not $d^{2}$). This is because by construction, $D_{0}$ is a \textbf{sparse matrix} such that its $ij$th element satisifies
  \begin{equation}
    \label{eq:structure-of-D0}
    D_{0,ij} = 0, \qquad \text{if } k(i), k(j) \in S_{0}.
  \end{equation}
\end{remark}

\subsection{A 3-Dimensional Example}
\label{sec:examples}

\begin{example}\label{example:ex1-3d}
  In this example, the system matrix $A$ and its Jordan canonical form are given as follows:
  \begin{equation}\label{eq:example-3d}
    \begin{gathered}
      A =
      \begin{pmatrix}
        0 & 1 & -1 \\
        2 & 0 & 0 \\
        3 & 1 & 0
      \end{pmatrix}
      = Q J Q^{-1}, \qquad
      J =
      \begin{pmatrix}
        -1 & 0 & 0 \\
        0 & 1/2 & \sqrt{7}/2 \\
        0 & -\sqrt{7}/2 & 1/2
      \end{pmatrix}. \\
      Q \approx
      \begin{pmatrix}
        0.408 & 0 & 0.316 \\
        -0.816 & 0.418 & 0.158 \\
        -0.408 & 0.837 & 0
      \end{pmatrix}, \qquad
      Q^{-1} \approx
      \begin{pmatrix}
        0.612 & -1.225 & 0.612 \\
        0.299 & -0.598 & 1.494 \\
        2.372 & 1.581 & -0.791
      \end{pmatrix}.
    \end{gathered}
  \end{equation}

  Based on earlier discussions, $A$ has two proper invariant subspaces. $L_{1} := \mathrm{span}(Q_{\cdot 1})$ is a one-dimensional space corresponding with the real eigenvalue $\lambda_{1}=-1$, and $L_{2} := \mathrm{span}\left( Q_{\cdot 2}, Q_{\cdot 3} \right)$ is a two-dimensional space corresponding with $\lambda_{2}, \lambda_{3} = 1/2 \pm \sqrt{7}i/2$. Here $Q_{\cdot j}$ is the $j$th column vector of matrix $Q$.

  Let us consider the following two initial conditions:
  \begin{equation*}
    \mathbf{x}_{0}^{(a)} := Q
    \begin{pmatrix}
      2 \\
      -1 \\
      0
    \end{pmatrix} \approx
    \begin{pmatrix}
      0.816 \\
      -2.051 \\
      -1.653
    \end{pmatrix}, \qquad
    \mathbf{x}_{0}^{(b)} := Q
    \begin{pmatrix}
      0 \\
      -2 \\
      3
    \end{pmatrix} \approx
    \begin{pmatrix}
      0.949 \\
      -0.362 \\
      -1.673
    \end{pmatrix}.
  \end{equation*}

  Notice that $\mathbf{x}_{0}^{(a)}$ contains information from both $L_{1}$ and $L_{2}$, but $\mathbf{x}_{0}^{(b)}$ only contains information from $L_{2}$. Based on earlier discussions, $A$ is identifiable at $\mathbf{x}_{0}^{(a)}$ but not $\mathbf{x}_{0}^{(b)}$. Using Equation~\eqref{eq:structure-of-equiv-class}, the $(A,\mathbf{x}_{0})$-unidentifiable class in the latter case can be represented as follows.
  \begin{equation}
    \label{eq:example1-equiv-class}
    \begin{gathered}
      I_{+} =
      \begin{pmatrix}
        0 & 0 & 0 \\
        0 & 1 & 0 \\
        0 & 0 & 1
      \end{pmatrix}, \quad I_{0} =
      \begin{pmatrix}
        1 & 0 & 0 \\
        0 & 0 & 0 \\
        0 & 0 & 0
      \end{pmatrix}, \\
      I_{0}DI_{0} =
      \begin{pmatrix}
        b & 0 & 0 \\
        0 & 0 & 0 \\
        0 & 0 & 0
      \end{pmatrix}, \quad b\in \R. \\
      J_{+} =
      \begin{pmatrix}
        0 & 0 & 0 \\
        0 & 1/2 & \sqrt{7}/2 \\
        0 & -\sqrt{7}/2 & 1/2
      \end{pmatrix}, \qquad Q (I_{0}DI_{0})Q^{-1} =
      bQ_{0}Q^{-1},  \\
      \begin{split}
        [A]_{\mathbf{x}_{0}^{(b)}} &= QJ_{+}Q^{-1} +bQ_{0}Q^{-1} \\
        &=
        \dfrac{1}{4}
        \begin{pmatrix}
          1 & 2 & -3 \\
          6 & 4 & -2 \\
          11 & 6 & -1
        \end{pmatrix}
        + \dfrac{b}{4}
        \begin{pmatrix}
          1 & -2 & 1 \\
          -2 & 4 & -2 \\
          -1 & 2 & -1
        \end{pmatrix}.
      \end{split}
    \end{gathered}
  \end{equation}

  Here $b\in \R$ is an arbitrary parameter, and the $(A,\mathbf{x}_{0})$-unidentifiable class is characterized by $b$ times a \emph{full matrix}, therefore different choices of $b$ affects all nine elements in $A$. Consequently, we cannot determine the value of \emph{any} entry in $A$ without additional information. In fact, we cannot even determine whether a particular $A_{ij}$ is zero or nonzero, which is sometimes referred to as the network topology of $A$. For example, if we set $b=-1$, we get the original system $A$ specified in Equation~\eqref{eq:example-3d} with four zero entries.  When we set $b=3$, we obtain an equivalent matrix with completely different topology and values than $A$
  \begin{equation}
    \label{eq:example1-Atilde}
    \tilde{A} =
    \begin{pmatrix}
      1 & -1 & 0 \\
      0 & 4 & -2 \\
      2 & 3 & -1
    \end{pmatrix}.
  \end{equation}

  It is easy to check that:
  \begin{equation*}
    % \label{eq:example1-conclusion}
    e^{t\tilde{A}}\mathbf{x}_{0}^{(a)} \neq e^{tA}\mathbf{x}_{0}^{(a)}, \qquad e^{t\tilde{A}}\mathbf{x}_{0}^{(b)} = e^{tA}\mathbf{x}_{0}^{(b)}.
  \end{equation*}
\end{example}

Before we move on to the next topic (practical identifiability), we would like to present two auxiliary results that are useful in practice.
\begin{enumerate}
\item In Supplementary Text, Section~\ref{sec:repeated-eigenvalues}, we provide a detailed analysis of identifiability issues induced by repeated eigenvalues in $A$, and provided the closed-form structure of unidentifiable class for these matrices in Equation~\eqref{eq:rep-eigen-equiv-class}. Based on these results,  we recommend researchers avoid systems that have nearly identical eigenvalues in designing of simulation studies.
\item in Supplementary Text, Section~\ref{sec:prior-inform-ident}, we show that while it is possible to use prior information in the form of structural constraints to resolve the identifiability issue of a problematic system $(A,\mathbf{x}_{0})$, such constraints must be \textbf{compatible} with the said system, otherwise: (a) the system may still suffer from the identifiability issue; or (b) under these constraints, no system can generate the observed solution curve.
\end{enumerate}

\section{Data-based Identifiability Scores}
\label{sec:pract-ident}

We now focus on the following practical problem: to quantify the $(A,\mathbf{x}_{0})$-identifiability from imperfect observations in real world applications. To this end, we assume that the observed data is a set of discrete and noisy observations on a time grid $\left\{t_{1}, \dots, t_{J}  \right\}$:
\begin{equation}
  \label{eq:noisy-obs}
  \begin{gathered}
    y_{ij} := x_{i}(t_{j}) + \epsilon_{ij}, \qquad \epsilon_{ij}\sim F_{\epsilon}, \qquad i=1,2,\dots, d, \quad j=1, 2, \dots, n.
  \end{gathered}
\end{equation}

In the above equation, $F_{\epsilon}$, the probability distribution of measurement error, is assumed to be absolutely continuous w.r.t. the Lebesgue measure on $M_{d\times J}$.  For convenience, we will also use collective notations $X = \left\{x_{i}(t_{j})\right\} \in M_{d\times n}$, $Y = \{y_{ij}\} \in M_{d\times n}$, and $\bm{\epsilon} = \{\epsilon_{ij}\} \in M_{d\times n}$. With these matrix notations, Equation~\eqref{eq:noisy-obs} can be simplified as $Y = X + \bm{\epsilon}$.

\subsection{Minimal Signals for Reconstructing $A$}
\label{sec:min-signal-for-A}
In this section, we demonstrate that even for a theoretically identifiable system, if the ``signal'' in a subspace is too small, we are still not able to reconstruct $A$ in practice.

\begin{example}\label{example:2d-hard}
  We consider a two-dimensional system
  \begin{equation}
    \label{eq:example2}
    \begin{gathered}
      A = Q \Lambda Q^{-1} =
      \begin{pmatrix}
        -1.875 & 2.382 \\
        2.382 & -4.625
      \end{pmatrix}
      , \\
      Q = \mathrm{Rot}_{\pi/6} =
      \begin{pmatrix}
        \frac{\sqrt{3}}{2} & -\frac{1}{2} \\
        \frac{1}{2} & \frac{\sqrt{3}}{2}
      \end{pmatrix}, \quad
      Q^{-1} = \mathrm{Rot}_{-\pi/6} =
      \begin{pmatrix}
        \frac{\sqrt{3}}{2} & \frac{1}{2} \\
        -\frac{1}{2} & \frac{\sqrt{3}}{2}
      \end{pmatrix}, \quad \Lambda =
      \begin{pmatrix}
        -1/2 & 0 \\
        0 & -6
      \end{pmatrix}. \\
    \end{gathered}
  \end{equation}

  We generate the solution curve $\mathbf{x}(t)$ from this system and record its values $x_{i}(t_{j})$ at $n=101$ equally spaced time points on $[0,1]$, $t_{1}=0, t_{2} = 0.01, \dots, t_{101}=1$. A small normal measurement error, $\epsilon_{ij} \sim N(0, 0.01^{2})$, is added to each observation.

  It is easy to see that $A$ has two one-dimensional proper invariant subspaces, $L_{1} = \mathrm{span}((\frac{\sqrt{3}}{2},  \frac{1}{2})')$, $L_{2} = \mathrm{span}((-\frac{1}{2}, \frac{\sqrt{3}}{2})')$. Let us consider two initial conditions
  \begin{equation*}
    \mathbf{x}_{0}^{(A)} =
    \begin{pmatrix}
      1 \\
      1
    \end{pmatrix}, \qquad
    \mathbf{x}_{0}^{(B)} =
    \begin{pmatrix}
      1.72 \\
      1
    \end{pmatrix}.
  \end{equation*}

  It can be shown that $A$ is identifiable at both $\mathbf{x}_{0}^{(A)}$ and $\mathbf{x}_{0}^{(B)}$. Of note, we would like to mention that this analysis can be done by applying the \texttt{ICISAnalysis()} function in our R package.

  Using the functional two-stage method (see Section~\ref{sec:functional-2stage}, Supplementary Text), we are able to estimate $A$ and produce two fitted curves for both cases.  The fitted curves, denoted by $(\hat{x}_{1}(t), \hat{x}_{2}(t))'$, look reasonable in Figure~\ref{fig:example2} for both cases.

  \begin{figure}
    \includegraphics[width=\textwidth]{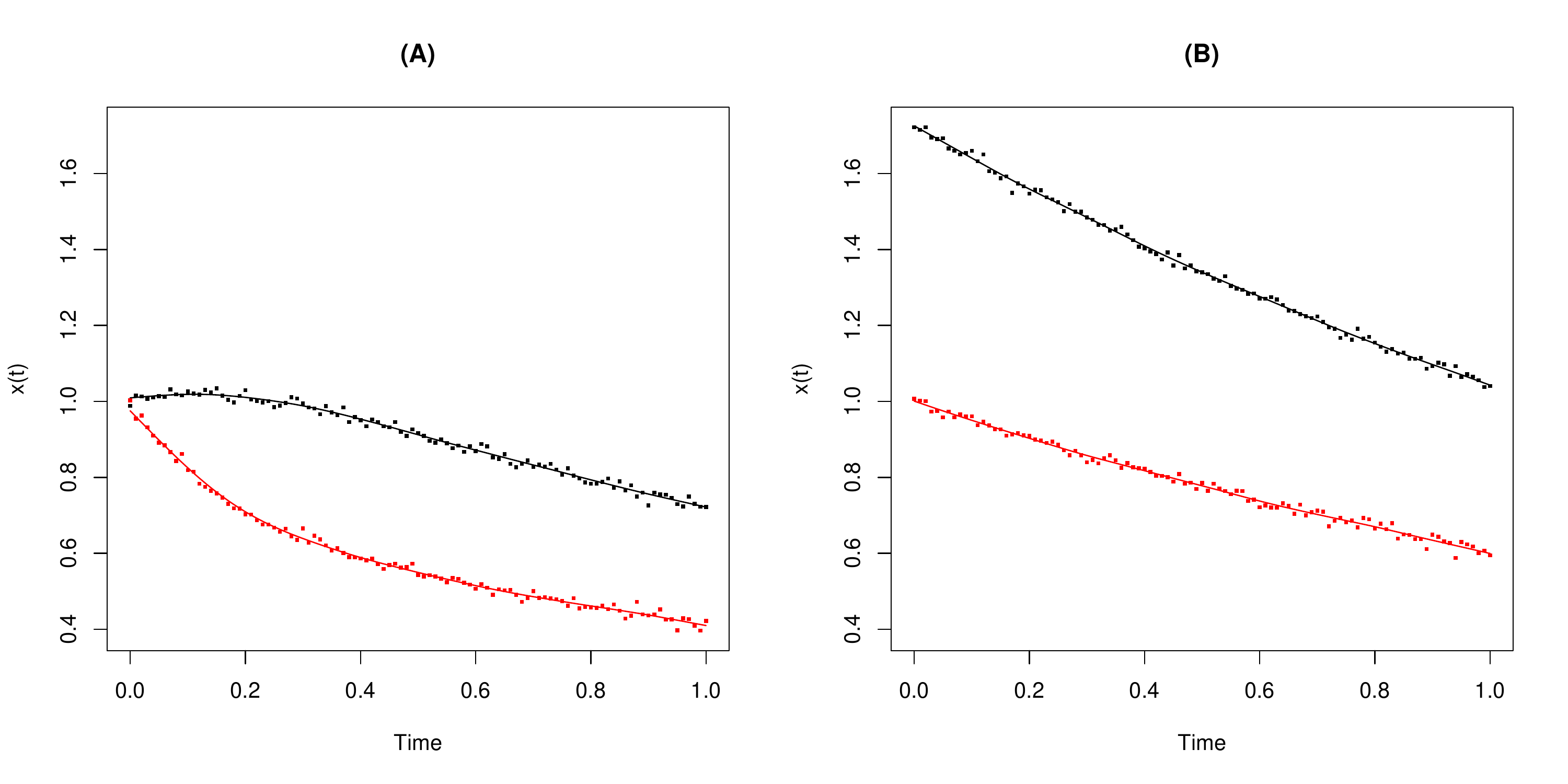}
    \caption{An illustration of the fitted solution trajectories in Example \ref{example:2d-hard}. The black curve is $\hat{x}_{1}(t)$ and the red curve is $\hat{x}_{2}(t)$ in both sub-figures. The discrete data (illustrated by dots in this figure) are observed on a total number of $n=101$ time points evenly assigned on $[0,1]$. A small normal measurement error, $\epsilon_{ij} \sim N(0, 0.01^{2})$, is added to each observation. (A): $\mathbf{x}_{0}^{(A)} = (1,1)'$.  (B): $\mathbf{x}_{0}^{(B)} = (1.72, 1)'$.  Case (A) is practically identifiable with the functional two-stage method while Case (B) is not.}
    \label{fig:example2}
  \end{figure}

  However, the two reconstructed system matrices are quite different:
  \begin{equation*}
    \begin{gathered}
      \hat{A}^{(A)} =
      \begin{pmatrix}
        -1.58 & 1.93 \\
        1.93 & -3.91
      \end{pmatrix},\qquad \|\hat{A}^{(A)} -A \|_{F}^{2} = 1.01, \\
      \hat{A}^{(B)} =
      \begin{pmatrix}
        -0.69 & 0.33 \\
        -1.81 & 2.62
      \end{pmatrix}, \qquad \|\hat{A}^{(B)} -A \|_{F}^{2}  = 75.65.
    \end{gathered}
  \end{equation*}

  Here $\|\cdot\|_{F}$ is the Frobenius norm of a matrix. It is clear that $A$ is practically identifiable at $\mathbf{x}_{0}^{(A)}$ only, not at $\mathbf{x}_{0}^{(B)}$.  From this example, we see that even when the ODE system is mathematical identifiable, its practical identifiability may still be an issue.

\end{example}

In-depth analysis shows that the unidentifiability issue in case (B) is due to the fact that $\mathbf{x}_{0}^{(B)}$ is ``almost'' contained in $L_{1}$, so that $L_{2}$ had only a tiny bit of information. Denote the basis in $L_{1}$ and $L_{2}$ as $Q_{1} = (\frac{\sqrt{3}}{2},  \frac{1}{2})'$ and $Q_{2} = (-\frac{1}{2}, \frac{\sqrt{3}}{2})'$, respectively. We have
\begin{equation*}
  \begin{gathered}
    w_{0,1}^{(A)} = \uinner{\mathbf{x}_{0}^{(A)}}{Q_{1}} = 1.366, \qquad \mathrm{ICIS}^{(A)} = w_{0,2}^{(A)} = \uinner{\mathbf{x}_{0}^{(A)}}{Q_{2}} = 0.366. \\
    w_{01}^{(B)} = \uinner{\mathbf{x}_{0}^{(B)}}{Q_{1}} = 1.990, \qquad
    \mathrm{ICIS}^{(B)} = w_{02}^{(B)} = \uinner{\mathbf{x}_{0}^{(B)}}{Q_{2}} = \mathbf{0.006}.
  \end{gathered}
\end{equation*}

Based on the above analysis, it is easy to see that the small value of $\mathrm{ICIS}^{(B)} = 0.006$ causes the numerical problem in estimating $A$ in case (B).

\subsection{Identifiability for High-dimensional Systems}
\label{sec:ident-highdim}

Based on Theorem~\ref{thm:identifiability}, $A$ is identifiable at $\mathbf{x}_{0}$ if and only if $\mathbf{x}_{0}$ is not located in a proper invariant subspace of $A$.  Because there are only finitely many ($2^{K}-1$ of them, to be more precise) proper invariant subspaces of $A$, and each of them has dimension strictly less than $d$ (the ``proper'' part of the definition), the union of all proper invariant subspaces is only a zero-measure set of $\R^{d}$.  In this regard, as long as $A$ does not have repeated eigenvalues (which is true for almost every $A \in M_{d\times d}$), $A$ is mathematically identifiable at almost every $\mathbf{x}_{0} \in \R^{d}$.  This fact is probably the main reason why not many mathematicians have paid much attention to the identifiability problem of linear ODE systems.

However, as is shown in Example~\ref{example:2d-hard}, to have a reliable estimate of $A$ requires more than just a \emph{qualitative} statement that $\mathbf{x}_{0}$ does not lie in any proper invariant subspace of $A$.  We need to ensure that when we decompose $\mathbf{x}_{0}$ into a linear combination of components from $L_{k}$, each one of them has \emph{enough} information, so that we can reconstruct the corresponding sub-system on $L_{k}$ with noisy observations. This is the main motivation for us to propose ICIS, a \emph{quantitative} measure of identifiability.

% As in Equation~\eqref{eq:linear-decomp-x0}, for a given $\mathbf{x}_{0}$ and the eigenvalue matrix $Q$, the oblique projection of $\mathbf{x}_{0}$ on $L_{k}$ is $w_{0,k}Q_{k}$, where $Q_{k}$ is either: (a) the eigenvector corresponding with a real eigenvalue, or (b) two basis vectors related to a pair of complex eigenvalues.  WLOG, we can assume that $Q_{k}$ are orthonormal, so $|Q_{k} w_{0,k}| = |w_{0,k}|$.  This number determines the ``size'' of the solution curve (without noise). If one or more of $|w_{0,k}|$ are too small, the solution curves in these $L_{k}$ would be very small, therefore the signal-to-noise ratio (SNR) will be too small to obtain good estimates of $A$ in these subspaces.

Knowing that the collection of all proper invariant subspace has measure zero in $\R^{d}$, the readers may think that while practical identifiability issues do exist, they must be rare in practice.  Unfortunately, these issues are not that unusual when $d$ is \textbf{large}, in which case those practically identifiable systems are the exceptions instead. In Supplementary Text, Section~\ref{sec:highdim-ode-unidentifiable}, we proved that a large class of symmetric random ODE systems are practically unidentifiable when $d\to \infty$, as stated in the following theorem
\begin{theorem}\label{thm:GOE-like}
  Let us assume that:
  \begin{enumerate}[(a)]
  \item The system matrix $A \in M_{d\times d}$ is sampled from a symmetric, real-valued random matrix ensemble with probability measure $p(A)$ that is statistically invariant to orthogonal transformations, namely,
    \begin{equation}
      \label{eq:orthogonal-invariant}
      p(A) = p(TAT'), \qquad \forall T \in O(d).
    \end{equation}
  \item The initial condition $\mathbf{x}_{0} \in \R^{d}$ is sampled from a random distribution that is independent of $A$ and satisfies
    \begin{equation}
      \label{eq:x0-growth-cond}
      \lim_{d\to \infty} \frac{E |\mathbf{x}_{0}|^{2}}{d^{3}} = 0.
    \end{equation}
  \end{enumerate}

  Based on the above two assumptions, the ICIS converges to zero in $L^{2}$, namely,
  \begin{equation}
    \label{eq:w0star-converges-to-zero}
    E \left(w_{0}^{*}(A, \mathbf{x}_{0})\right)^{2} \longrightarrow 0, \quad \text{when $d \to \infty$.}
  \end{equation}
\end{theorem}
\begin{proof}
  The proof is provided in Section~\ref{sec:highdim-ode-unidentifiable}, Supplementary Text.
\end{proof}

\begin{remark}
  \label{remark:GOE-assumptions}
  Perhaps the most well known random matrix ensemble that satisfies Assumption (a) is the Gaussian Orthogonal Ensemble (GOE, \cite{tao2012topics}). Many other ensembles also satisfy this condition, such as the Wishart ensemble, Jacobi orthogonal ensemble, etc. In fact, according to Weyl's lemma~\cite{livan2018introduction,weyl1946classical}, a random matrix ensemble is orthogonally invariant as long as its distribution function has the following trace representation
  \begin{equation}
    \label{eq:Weyl-lemma}
    p(H) = \phi \left( \tr H, \tr H^{2}, \dots, \tr H^{n} \right).
  \end{equation}

  Assumption (b) is a very weak condition that should be satisfied in almost all practical applications. If $\mathbf{x}_{0,i}$ has finite second order moments, and
  \begin{equation*}
    \sup_{i} Ex_{0,i}^{2} = \mu_{2} < \infty,
  \end{equation*}

  we have
  \begin{equation*}
    E|\mathbf{x}_{0}|^{2} := \sum_{i=1}^{d} Ex_{0,i}^{2} \leqslant d\cdot \mu_{2} = O(d^{1}), \qquad \frac{E |\mathbf{x}_{0}|^{2}}{d^{3}} = O(d^{-2}) \to 0.
  \end{equation*}

  In this case, $x_{0,i}$ do not have to be independent nor identically distributed. 
\end{remark}

We use the following simple and concrete example to illustrate the issue of practical identifiability described by Theorem~\ref{thm:GOE-like}.
\begin{example}\label{example:d100}
  Let $d=100$ and assume that $A$ is an arbitrary diagonal matrix in $M_{d\times d}$. By construction, all eigenvalues are real and $Q = I_{d}$, therefore $w_{0,i}=x_{0,i}$, $i=1,2,\dots, d$. Let $\mathbf{x}_{0} \sim N(0_{d}, I_{d})$, in other words, $x_{0,i}$ are generated from $i.i.d.$ $N(0,1)$. For simplicity, we write $r_{i} = |w_{0,i}|$. Apparently, $r_{i}$ are standard half normals with relatively large expectations $E r_{i} = \sqrt{\frac{2}{\pi}} \approx 0.8$. This fact seems to suggest that, as long as the measurement error is small ($\sigma \ll 0.8$), we would have enough information to reconstruct $A$.

  In reality, the \emph{smallest} member of $r_{i}$ (denoted by $r_{(1)}$), has a distribution that is statistically much smaller than a standard half normal. Using numerical integration, we found that $Er_{(1)} \approx 0.012$, which is \textbf{66 times smaller} than $E r_{i}$. Based on the lessons we learned from Example~\ref{example:2d-hard}, we anticipate that it is almost impossible to estimate $A$ accurately in this case. 

Finally, we conduct a mini-simulation to illustrate that ODEs with random \emph{asymmetric} system matrices also suffer from the identifiability issues stated in Theorem~\ref{thm:GOE-like}.  Specifically, we randomly generate 50 $A$ from the standard GinOE with $d=3,5,100$ dimensions, and pair them with 50 $\mathbf{x}_{0}$ sampled from $N(0_{d}, I_{d})$. We compute the ICIS for those $(A,\mathbf{x}_{0})$ and plot them in Figure~\ref{fig:Dimension-w0star}. We see that larger dimension is associated with smaller ICIS, which implies that these ODE systems are more difficult to be numerically identified.
\end{example}

\begin{figure}[ht]
  \centering
  \includegraphics[width=.7\textwidth]{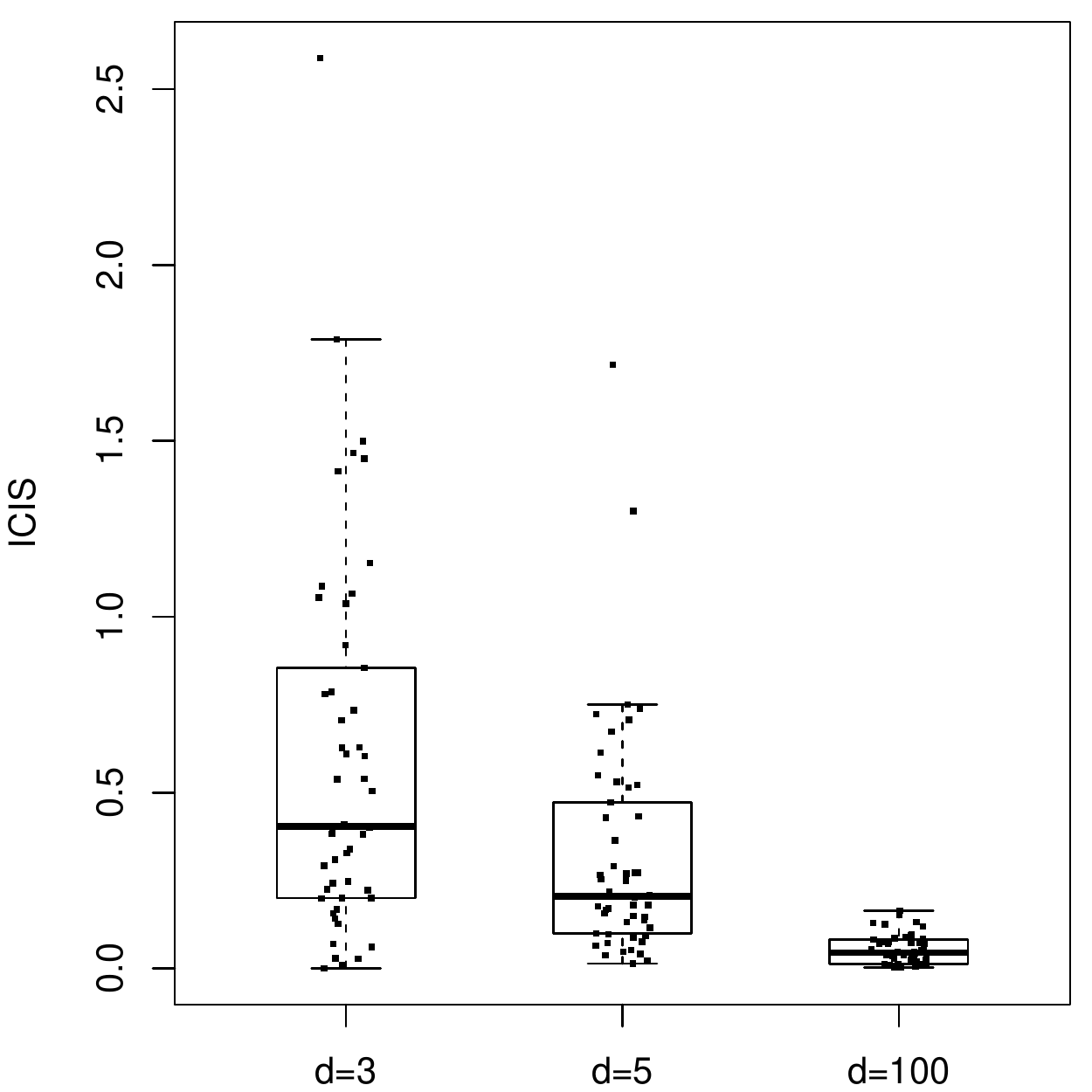}
  \caption{Larger dimension is associated with smaller ICIS ($\wa$), which implies that high-dimensional ODE systems are more difficult to be reconstructed numerically.}
  \label{fig:Dimension-w0star}
\end{figure}

In light of the above discussions, to design a well-behaved, identifiable high-dimensional ODE system in simulations, one needs to ensure that:
\begin{enumerate}
\item $A$ has nice mathematical properties, such as distinct eigenvalues; stability; no high-frequency components, etc.
\item $\mathbf{x}_{0}$ should not be ``randomly'' generated; instead, it should be generated in a way such that $\mathrm{ICIS}(A,\mathbf{x}_{0})$ is not too small.
\end{enumerate}

\subsection{Practical Identifiability Score (PIS)}
\label{sec:w-statistic}

Recall that ICIS does not depend on the full trajectory $\mathbf{x}(t)$, therefore it is most useful in designing simulation studies.  In real world applications, it is preferable to define identifiability scores that use the entirety of $Y$ (the discrete data measured at all time points) to quantify the practical identifiability of the system. Let $\hat{A}(Y)$ be an estimator of $A$ given the discrete observation. One way to quantify the practical identifiability is to use the numerical sensitivity of $\hat{A}$, which can be defined as the mean squared error, $\mathrm{MSE} := E\|\hat{A} - A\|_{F}^{2}$. Unfortunately, there are many different ways to estimate $A$, thus it is impossible to develop a \emph{universal} quantity that works for all estimators.  In this section, we propose two scores based on a class of the two-stage methods, under the assumption that $\sigma^{2}$ is small enough so that $\hat{A} - A$ is small. These proposed scores are compared to $\kappa(Y_{1})$, the practical identifiability measure proposed by Stanhope and colleagues in~\cite{stanhope2014identifiability} in our simulation studies.

\subsubsection{Stanhope's Condition Number}
\label{sec:stanhope-cond-num}

Stanhope and colleagues proposed to use $\kappa(Y_{1})$ as a practical identifiability measure in~\cite{stanhope2014identifiability}. Here $Y_{1} := [Y_{\cdot 1}| Y_{\cdot 2} | \dots | Y_{\cdot d}] \in M_{d\times d}$ is the matrix of discrete data evaluated at the first $d$ time points; $\kappa(Y_{1}) := \|Y_{1}\|_{F} \|Y_{1}^{-1}\|_{F}$ is the \emph{condition number} of $Y_{1}$, which is a quantitative measure of numerical stability of $Y_{1}^{-1}$.  It was motivated by the fact that if $Y_{1}$ is invertible and there is no noise in the discrete observations,
\begin{equation}
  \label{eq:stanhope-kappa-motivation}
  e^{\Delta t A} = Y_{2} Y_{1}^{-1}, \qquad Y_{2} := [Y_{\cdot 2}| Y_{\cdot 3} | \dots | Y_{\cdot d+1}] \in M_{d\times d}.
\end{equation}

\subsubsection{Functional Two-stage Methods}
\label{sec:two-stage-methods}

% Below we give a precise definition of the class of two-stage methods used in our study.

\begin{definition}[pairwise $L^{2}$-inner product matrix]
  Let $\mathbf{x}(t) \in \R^{d}$ and $\mathbf{y}(t) \in \R^{d'}$ be two multidimensional functions defined on $[0, T]$. We use notation $\uinner{\mathbf{x}(t)}{\mathbf{y}(t)}$ to refer to the pairwise $L^{2}$-inner product matrix between them, which is a $d\times d'$ matrix in which
  \begin{equation*}
    \uinner{\mathbf{x}(t)}{\mathbf{y}(t)}_{i,j} := \uinner{x_{i}(t)}{y_{j}(t)}_{L^{2}} = \int_{0}^{T} x_{i}(t)y_{j}(t)\ud t.
  \end{equation*}
\end{definition}

For convenience, we denote $\uinner{\mathbf{x}(t)}{\mathbf{x}(t)}$ by $\Sigma_{\mathbf{x}\mathbf{x}}$ when there is no confusion. Obviously, $\Sigma_{\mathbf{x}\mathbf{x}}$ is a symmetric positive semi-definite matrix. It is singular if and only if zero is one of its eigenvalues, in which case we know
\begin{equation}
  \label{eq:zero-eigenvector-for-pairwise-inner-prod}
  \mathbf{v}_{0}' \Sigma_{\mathbf{x}\mathbf{x}} \mathbf{v}_{0} = 0, \qquad \mathbf{v}_{0} \in \R^{d}, \quad \mathbf{v}_{0}\ne 0_{d}.
\end{equation}
where $\mathbf{v}_{0}$ is an eigenvector associated with the zero eigenvalue.

\begin{theorem}\label{thm:identifiability-based-on-pairwise-inprod}
  Let $\mathbf{x}(t)$ be an observed solution trajectory governed by ODE system $A$ and initiated at $\mathbf{x}_{0} := \mathbf{x}(0)$. This ODE system is identifiable at $\mathbf{x}_{0}$ if and only if the pairwise inner product matrix $\Sigma_{\mathbf{x}\mathbf{x}}$ is nonsingular (invertible).
\end{theorem}
\begin{proof}
  See Section~\ref{sec:mathematical-proofs}, Supplementary Text.
\end{proof}

Notice that the ODE $D\mathbf{x}(t) = A\mathbf{x}(t)$ implies
\begin{equation}
  \label{eq:two-stage-motivation}
  \uinner{D\mathbf{x}(t)}{\mathbf{x}(t)} = A \uinner{\mathbf{x}(t)}{\mathbf{x}(t)} \Longrightarrow A = \Sigma_{D\mathbf{x}, \mathbf{x}} \cdot \Sigma_{\mathbf{x}\mathbf{x}}^{-1}.
\end{equation}

Here $\Sigma_{D\mathbf{x}, \mathbf{x}} := \uinner{D\mathbf{x}(t)}{\mathbf{x}(t)}$ is the pairwise inner product matrix between $D\mathbf{x}(t)$ and $\mathbf{x}(t)$. This fact motivated the following two-stage methods, which is a class of estimators of $A$:
\begin{equation}
  \label{eq:two-stage-est}
  \hat{A} := \hat{\Sigma}_{D\mathbf{x}, \mathbf{x}} \cdot  \hat{\Sigma}_{\mathbf{x}\mathbf{x}}^{-1}.
\end{equation}

Here $\hat{\Sigma}_{D\mathbf{x}, \mathbf{x}}$ and $\hat{\Sigma}_{\mathbf{x}\mathbf{x}}$ are estimates of $\Sigma_{D\mathbf{x}, \mathbf{x}}$ and $\Sigma_{\mathbf{x}\mathbf{x}}$, respectively. There are many choices of these estimators, some of which include tuning parameter(s).  For example, one can use discretizing techniques and finite differences to estimate them. We call this approach the \emph{simple two-stage method}. Alternatively, we could use roughness penalized basis splines to estimate $\hat{\mathbf{x}}(t)$, then apply the differential operator and integral operator to estimate those terms. We call the latter approach the \emph{functional two-stage method}. These two methods are described in detail in Section~\ref{sec:more-2stage}, Supplementary Text. In either case, $\hat{\Sigma}_{D\mathbf{x}, \mathbf{x}}$ and $\hat{\Sigma}_{\mathbf{x}\mathbf{x}}$ could be represented by the following matrix operations
\begin{equation}
  \label{eq:est-inprod-matrices}
  \hat{\Sigma}_{\mathbf{x}\mathbf{x}} = Y S Y', \quad \hat{\Sigma}_{D\mathbf{x}, \mathbf{x}} = Y L Y', \quad \hat{A} = YLY' (YSY')^{-1}.
\end{equation}
Here $S$ and $L$ were two $(n\times n)$-dimensional matrices obtained from the particular estimating procedure, such as the smoothing step in the functional two-stage methods.

% For simplicity, we assume that $S$ is symmetric, which is the case for most practical algorithms.

% The following theorem states that, when there are infinite observations, the functional two-stage method is an \textbf{exact estimator} of $A$, therefore it is reasonable to assume that $\hat{A} - A$ is small in our subsequent analyses.

% \begin{theorem}\label{thm:functional-two-stage-perfect}
%   Assuming that the pairwise $L^{2}$-inner product matrix $\Sigma_{\mathbf{x}\mathbf{x}}$ is invertible. Matrix $A$ is the system matrix if and only if it satisfies
%   \begin{equation}
%     \label{eq:two-stage-equation}
%     A = \Sigma_{D\mathbf{x}, \mathbf{x}} \cdot \Sigma_{\mathbf{x}\mathbf{x}}^{-1}.
%   \end{equation}
%   Furthermore, when $n\to \infty$, we have
%   \begin{equation}
%     \label{eq:convergence}
%     \hat{\Sigma}_{D\mathbf{x}, \mathbf{x}} \to \Sigma_{D\mathbf{x}, \mathbf{x}}, \qquad \hat{\Sigma}_{\mathbf{x}\mathbf{x}}^{-1} \to \Sigma_{\mathbf{x}\mathbf{x}}^{-1}, \qquad \hat{A} \to A.
%   \end{equation}
% \end{theorem}
% \begin{proof}
%   The proof is provided in Appendix~\ref{sec:mathematical-proofs}.
% \end{proof}

The following two theorems state that: a) when there is no error in estimating $\mathbf{x}(t)$, $\hat{A}$ defined in Equation~\eqref{eq:two-stage-est} is \textbf{exact}, and b) small errors in estimating $\mathbf{x}(t)$ and $D\mathbf{x}(t)$ only induce a small error in $\hat{A}$ for an identifiable $(A,\mathbf{x}_{0})$.

\begin{theorem}\label{thm:functional-two-stage-perfect}
Assume that a linear ODE system with constant coefficient is identifiable at $\mathbf{x}_{0}$. A matrix $A \in M_{d\times d}$ is its system matrix if and only if it satisfies
  \begin{equation}
    \label{eq:two-stage-equation}
    A = \Sigma_{D\mathbf{x}, \mathbf{x}} \cdot \Sigma_{\mathbf{x}\mathbf{x}}^{-1}.
  \end{equation}
\end{theorem}

\begin{theorem}\label{thm:functional-two-stage-continuous}
  Let $\hat{\mathbf{x}}(t)$ and $D\hat{\mathbf{x}}(t)$ be the estimates of the solution trajectory and its derivative used in Equation~\eqref{eq:two-stage-est}. Let $\delta_{1} := \|\hat{\mathbf{x}}(t) - \mathbf{x}(t)\|_{L^{2}}$ and $\delta_{2} := \|D\hat{\mathbf{x}}(t) - D\mathbf{x}(t)\|_{L^{2}}$ be the estimation errors measured in $L^{2}$ norm of $\hat{\mathbf{x}}(t)$ and $D\hat{\mathbf{x}}(t)$, respectively; and define $\delta := \max(\delta_{1}, \delta_{2})$.  We have
  \begin{equation}
    \label{eq:small-L2-error-induce-small-MSE}
    \|\hat{A} - A\|_{F} \leqslant C\delta.
  \end{equation}
  Here $C$ is a multiplicative constant that depends on $\|\hat{\mathbf{x}}(t)\|$, $\|D\hat{\mathbf{x}}(t)\|$, and the condition number of $\Sigma_{\mathbf{x}\mathbf{x}}$.
\end{theorem}
\begin{proof}
  The proofs of the above two theorems are provided in Section~\ref{sec:mathematical-proofs}, Supplementary Text.
\end{proof}

It is well known that, with reasonable knot placement and design points ($t_{j}$s), $\hat{x}(t)$ and $D\hat{x}(t)$ obtained by roughness penalized smoothing splines converge to $\mathbf{x}(t)$ and $D\mathbf{x}(t)$ in $L^{2}$ norm. Given Theorem~\ref{thm:functional-two-stage-perfect}, it is reasonable to assume that $\hat{A} - A$ is small in our subsequent analyses.

\subsubsection{Smoothed Condition Number}
\label{sec:smooth-cond-numb}

We first propose a straightforward generalization of Stanhope's condition number, called the smoothed condition number (SCN), to measure practical identifiability:
\begin{equation}
  \label{eq:tau-def}
  \tau(Y, S) := \kappa(\hat{\Sigma}_{\mathbf{x}\mathbf{x}}) = \kappa(YSY').
\end{equation}

Apparently, Equation~\eqref{eq:two-stage-est} is the main motivation of this generalization. Compared with Stanhope's $\kappa$ statistic, SCN incorporates the information contained in the smoothing operator $S$, therefore captures more information of the parameter estimation procedure.

\subsubsection{Practical Identifiability Score}
\label{sec:pis}

Based on Equation~\eqref{eq:est-inprod-matrices}, a small perturbation of data could induce a small $\|\hat{A} - A\|_{F}^{2}$. To conduct a formal sensitivity analysis, we need to make the following additional assumptions:
\begin{enumerate}
\item Measurement errors are uncorrelated: $\mathrm{cor}(\epsilon_{ij}, \epsilon_{i'j'}) = 0$ if $i\ne i'$ or $j\ne j'$.
\item These errors are relatively small, namely, $\mathrm{var}(\epsilon_{ij}) \leqslant \sigma^{2} \ll 1$ for all $i,j$.
\item $\|\hat{A}(X) - A\|_{F}^{2} \ll \|\hat{A}(Y) - A\|_{F}^{2}$, namely, the \emph{numerical error} in $\hat{A}$ due to the use of discrete data is much smaller than the \emph{variance} of $\hat{A}$ caused by measurement error. This assumption can also be expressed as $A \approx XLX'(XSX')^{-1}$, which is a reasonable assumption for cases in which $n$ is large based on Theorem~\ref{thm:functional-two-stage-perfect}.
\end{enumerate}

Denote $\bm{\epsilon}_{S} = \bm{\epsilon}SX' +XS\bm{\epsilon}'$, $\bm{\epsilon}_{L} = \bm{\epsilon}LX' +XL\bm{\epsilon}'$, and $N = \left( X S X' \right)^{-1}$. Based on the above assumptions, we have
\begin{equation}
  \label{eq:Ahat-approx1}
  \begin{split}
    \hat{A}(X+\bm{\epsilon}) -A &= (X+\bm{\epsilon})L(X+\bm{\epsilon})' \left( (X+\bm{\epsilon})S(X+\bm{\epsilon})' \right)^{-1} -A \\
    &\approx (XLX' +\bm{\epsilon}LX' +XL\bm{\epsilon}') \left( XSX' +\bm{\epsilon}SX' +XS\bm{\epsilon}' \right)^{-1} -A \\
    &\approx (XLX' +\bm{\epsilon}_{L}) \left( N -N \bm{\epsilon}_{S} N \right) -A \\
    &\approx -XLX'N(\bm{\epsilon}_{S})N +\bm{\epsilon}_{L}N \\
    &\approx (\bm{\epsilon}_{L} -A\bm{\epsilon}_{S})N.
  \end{split}
\end{equation}

Using Proposition~\ref{thm:mean-random-matrix}, we have
\begin{equation}
  \label{eq:MSE-Ahat-aux1}
  \small
  \begin{gathered}
    \begin{split}
      E (\bm{\epsilon}_{L}'\bm{\epsilon}_{L}) &= E ((\bm{\epsilon}L'X' +XL'\bm{\epsilon}') (\bm{\epsilon}LX' +XL\bm{\epsilon}')) \\
      &= E\left( \bm{\epsilon}L'X'\bm{\epsilon}LX' + \bm{\epsilon}L'X'XL\bm{\epsilon}' +XL'\bm{\epsilon}'\bm{\epsilon}LX' +XL'\bm{\epsilon}'XL\bm{\epsilon}' \right) \\
      &= \sigma^{2} \left(XL^{2}X' + \mathrm{tr}(L'X'XL)\cdot I_{d} + d\cdot XL'LX' +X(L')^{2}X'\right). \\
    \end{split} \\
    \begin{split}
      E \left( \bm{\epsilon}_{S}' A'\bm{\epsilon}_{L}  \right) &= E ((\bm{\epsilon}S'X' +XS'\bm{\epsilon}') A' (\bm{\epsilon}LX' +XL\bm{\epsilon}')) \\
      &= E \left( \bm{\epsilon}S'X'A'\bm{\epsilon}LX' +XS'\bm{\epsilon}'A'\bm{\epsilon}LX' +\bm{\epsilon} S'X'A'XL\bm{\epsilon}' +XS'\bm{\epsilon}'A'XL\bm{\epsilon}' \right) \\
      &\leqslant \sigma^{2}\left( AXSLX' + \mathrm{tr}(A)\cdot XS'LX' +\mathrm{tr}(S'X'A'XL)\cdot I_{d} +XS'L'X'A \right).
    \end{split} \\
    \begin{split}
      E\left( \bm{\epsilon}_{S}'A'A\bm{\epsilon}_{S} \right) &= E ((\bm{\epsilon}S'X' +XS'\bm{\epsilon}') A'A (\bm{\epsilon}SX' +XS\bm{\epsilon}')) \\
      &= E \big( \bm{\epsilon}S'X'A'A \bm{\epsilon} SX' +XS'\bm{\epsilon}'A'A \bm{\epsilon}SX' \\
      &\quad +\bm{\epsilon} S'X'A'AXS\bm{\epsilon}' +XS'\bm{\epsilon}'A'AXS\bm{\epsilon}' \big) \\
      &\leqslant \sigma^{2} \big(A'AXS^{2}X' +\mathrm{tr}(A'A)\cdot XS'SX' \\
      &\quad \mathrm{tr}\left( S'X'A'AXS \right)\cdot I_{d} +X (S')^{2} X'A'A \big).
    \end{split}
  \end{gathered}
\end{equation}

Therefore
\begin{equation}
  \label{eq:dMSE-Ahat}
  \begin{split}
    \|\hat{A}(X+\bm{\epsilon}) - A\|_{F}^{2} &\approx \mathrm{tr}\left( (\bm{\epsilon}_{L} -A\bm{\epsilon}_{S})N^{2} (\bm{\epsilon}_{L}' -\bm{\epsilon}_{S}'A') \right) \\
    &= \mathrm{tr}\left( (\bm{\epsilon}_{L}' -\bm{\epsilon}_{S}'A')  (\bm{\epsilon}_{L} -A\bm{\epsilon}_{S})N^{2} \right) \\
    &= \mathrm{tr}\left( (\bm{\epsilon}_{L}' \bm{\epsilon}_{L} -\bm{\epsilon}_{S}' A'\bm{\epsilon}_{L} -\bm{\epsilon}_{L}' A\bm{\epsilon}_{S} +\bm{\epsilon}_{S}'A'A\bm{\epsilon}_{S})N^{2} \right) \\
    E \|\hat{A}(X+\bm{\epsilon}) - A\|_{F}^{2} &\approx \mathrm{tr} \big( ( E(\bm{\epsilon}_{L}'\bm{\epsilon}_{L}) -2E \left( \bm{\epsilon}_{S}' A'\bm{\epsilon}_{L}  \right) +E\left( \bm{\epsilon}_{S}'A'A\bm{\epsilon}_{S} \right))N^{2} \big) \\
    &\leqslant \sigma^{2} W(X). \\
  \end{split}
\end{equation}

Here
\begin{equation}
  \label{eq:Wfun-def}
  \begin{split}
    W(X|A,S,L) &:= \mathrm{tr}\big((XSX')^{-2} \big( XL^{2}X' +d\cdot XL'LX' +X(L')^{2}X' \\
    &\quad -2AXSLX' -2\mathrm{tr}(A)\cdot XS'LX' -2XS'L'X'A \\
    &\quad +A'AXS^{2}X' +\mathrm{tr}(A'A)\cdot XS'SX' +X (S')^{2} X'A'A) \\
    &\quad + \mathrm{tr}(L'X'XL -2S'X'A'XL +S'X'A'AXS)\cdot I_{d} \big) \\
    &= \mathrm{vec}((XS'X')^{-2})' \mathrm{vec}\big( XL^{2}X' +d\cdot XL'LX' +X(L')^{2}X' \\
    &\quad -2AXSLX' -2\mathrm{tr}(A)\cdot XS'LX' -2XS'L'X'A \\
    &\quad +A'AXS^{2}X' +\mathrm{tr}(A'A)\cdot XS'SX' +X (S')^{2} X'A'A \big) \\
    &\quad +\mathrm{tr}(L'X'XL -2S'X'A'XL +S'X'A'AXS) \mathrm{tr}(N^{2}).
  \end{split}
\end{equation}

By construction, $W(X)$ is a scalar that quantifies the MSE of $\hat{A}$ as a function of $\sigma^{2}$, the maximum variance of $\epsilon_{ij}$ for all $i=1,2,\dots, d$ and $j=1,2,\dots, n$.  Smaller values of $W(X)$ imply better practical identifiability in reconstructing $A$. Motivated by this fact, we define the practical identifiability score (PIS) as the sample version of $W(X)$. Specifically, PIS (denoted as $w^{*}$ in Equation~\eqref{eq:wstar-def}) is computed by replacing $X$ and $A$ in Equation~\eqref{eq:Wfun-def} with $Y$ and $\hat{A} := YLY' (YSY')^{-1}$, respectively:
\begin{equation}
  \label{eq:wstar-def}
  \wb(Y|S,L) := W(Y|\hat{A},S,L).
\end{equation}

Compared with Stanhope's $\kappa$ and SCN, PIS depends not only on the observed data ($Y$), but also the $S$ and $L$ matrix of the particular two-stage method used in reconstructing $\hat{A}$, therefore it is a more accurate indicator of practical identifiability. A simulation study was designed to demonstrate this point in Section~\ref{sec:wstar-sim2}.

% Of note, if $YSY'$, which is an estimate of $\Sigma_{\mathbf{x}\mathbf{x}}$, is singular or near singular, PIS either is not well define or has very large value. This is expected because the inverse of $YSY'$ is used in the two-stage estimator.

\section{Simulation studies}
\label{sec:simulation-studies}

% \begin{enumerate}
% \item Xing's small simulations [DONE].
% \item Tao Xu's simulations:
%   \begin{enumerate}
%   \item Ginibre $A$; multivariate normal $\mathbf{x}_{0}$; low dimensional.
%   \item Ginibre $A$; multivariate normal $\mathbf{x}_{0}$; high-dimensional.
%   \item $A$ with well-designed eigen-structure.  Study the statistical relationship between $\mathbf{w}_{0}$ and identifiability.
%   \item Relationship between \hl{sparsity and connectivity} to identifiability.
%   \end{enumerate}
% \end{enumerate}

\subsection{ICIS is Inversely Correlated with the  Relative Estimation Error (REE)}
\label{sec:w0-sim1}

\begin{figure}
  \includegraphics[width=\textwidth]{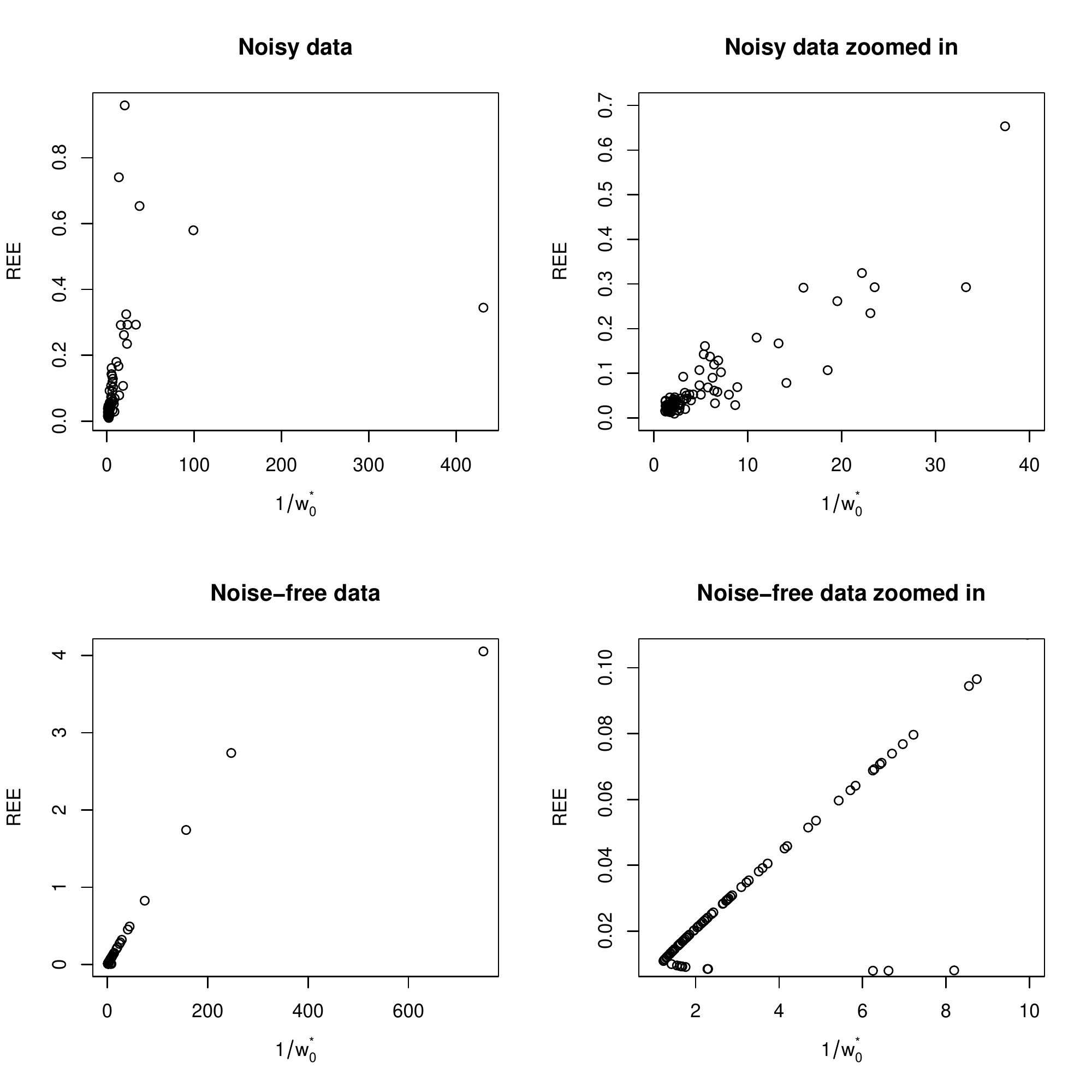}
  \caption{ICIS ($\wa$) is inversely correlated with the relative estimation error (REE). In all four subfigures, the $x$-axis is $1/\wa$, $y$-axis is REE. Each dot represents one of the 100 repetitions of SIM1. Top two subfigures are generated from noisy data ($Y$); bottom two subfigures uses noise-free data ($X$). The right subfigures are the zoomed-in version of the left subfigures. }
  \label{fig:SIM1-REE-w0star}
\end{figure}

We design SIM1 to demonstrate that the ICIS is inversely correlated with estimation error, as predicted in Section~\ref{sec:ident-highdim}.  In this simulation, the system matrix and its two invariant subspaces are:
\begin{equation*}
  A =
  \begin{pmatrix}
    -0.1 & 3 & 0 \\
    -3 & -0.1 & 0 \\
    0 & 0 & -0.5
  \end{pmatrix}, \qquad L_{1} = \mathrm{span}\left( \mathbf{e}_{1}, \mathbf{e}_{2} \right), \quad L_{2} = \mathrm{span}\left( \mathbf{e}_{3} \right).
\end{equation*}

Here $\mathbf{e}_{k}$ is the $k$th column vector of $I_{3}$, a.k.a. the $k$th natural basis vector of $\R^{3}$.  We generate $\mathbf{x}_{0}$ from $N(0_{3}, I_{3})$ first, then standardize it to have unit length to reduce the variation in ICIS due to different $|\mathbf{x}_{0}|$. This is equivalent to sampling $\mathbf{x}_{0}$ from a uniform distribution on $S^{2}$.

Once $\mathbf{x}_{0}$ is generated, we compute $X_{\cdot j} := e^{t_{j}A}\mathbf{x}_{0}$ at a time grid with range $[0, 6]$ and step $\Delta t=0.1$, i.e., $t_{j} = 0, 0.1, \dots, 6$ for $j=1,\dots, 61$. We also add a small noise $\epsilon_{ij} \sim N(0, \sigma^{2})$, $\sigma=0.05$ to each observation to create noisy data $Y = X + \bm{\epsilon}$.  A functional two-stage method based on cubic splines with roughness penalty $\lambda=0.001$ is used to estimate $\hat{A}$ from both noisy ($Y$) and noise-free ($X$) data.

The accuracy of estimation is measured by relative estimation error (REE),  defined as follows
\begin{equation}
  \label{eq:REE-def}
  \mathrm{REE}(\hat{A}, A) := \frac{\|\hat{A}-A\|_{F}}{\|A\|_{F}}.
\end{equation}

We repeat SIM1 for 100 times, each with randomly generated $\mathbf{x}_{0}$ and measurement error. We find that ICIS ($\wa$) is strongly negatively correlated with REE. The Spearman correlation between these two quantities is $\rho_{1}= -0.803$ for the noisy data and $\rho_{2}=-0.843$ for the noise-free data.  This inverse correlation is visualized in Figure~\ref{fig:SIM1-REE-w0star}. Other than a few outliers, $1/\wa$ has an almost perfect linear relationship with $\mathrm{REE}$ in the noise-free case (the second row of Figure~\ref{fig:SIM1-REE-w0star}).  The correlation between $1/\wa$ and $\mathrm{REE}$ is weaker but still quite apparent for the noisy data (the first row of Figure~\ref{fig:SIM1-REE-w0star}).

\subsection{Using SCN and PIS to Classify Identifiable and Un-identifiable Systems}
\label{sec:wstar-sim2}

We design SIM2 to demonstrate that, when data collected at all time points are available, SCN and PIS have better performance in classifying identifiable and un-identifiable Systems than ICIS and Stanhope's $\kappa$. SIM2 contains one identifiable case and two unidentifiable cases, which are described as follows.
\begin{enumerate}
\item In each one of 200 repetitions, we generate two $4\times 4$-dimensional system matrices $A$ and $B$, and two initial conditions $\mathbf{x}_{0}^{(a)}$ and $\mathbf{x}_{0}^{(b)}$ on $S^{3}$.
\item Both $A$ and $B$ have one pair of complex eigenvalues and two real eigenvalues. The eigenvalues of $A$, $(\lambda_{1}, \lambda_{2}, \lambda_{3}, \lambda_{4})$, are generated in this way
  \begin{equation}
    \label{eq:SIM2-eigenvalues}
    \begin{gathered}
      \lambda_{1},\lambda_{2} = -0.1 \pm bi, \quad b\sim \mathrm{Unif}([2,4]). \\ \lambda_{3} \sim \mathrm{Unif}([-0.8, -0.4]), \quad \lambda_{4} \sim \mathrm{Unif}([-2, -1.2]).
    \end{gathered}
  \end{equation}
  The eigenvalues of $B$ are set to be $(\lambda_{1}, \lambda_{2}, \lambda_{3}, \lambda_{3})$, namely, $B$ has a pair of \emph{repeated eigenvalues} by construction. Therefore it is not identifiable at any initial condition.
\item After generating the eigenvalues, we sample an orthogonal matrix $Q$ from the standardized Haar measure (the uniform distribution) on the orthogonal group $O(4)$, and create two system matrices $A$ and $B$ as follows:
  \begin{equation}
    \label{eq:SIM2-A-B}
    A = Q
    \begin{pmatrix}
      -0.1 & b & 0 & 0 \\
      -b & -0.1 & 0 & 0 \\
      0 & 0 & \lambda_{3} & 0 \\
      0 & 0 & 0 & \lambda_{4} \\
    \end{pmatrix} Q', \quad
    B = Q
    \begin{pmatrix}
      -0.1 & b & 0 & 0 \\
      -b & -0.1 & 0 & 0 \\
      0 & 0 & \lambda_{3} & 0 \\
      0 & 0 & 0 & \lambda_{3} \\
    \end{pmatrix} Q'.
  \end{equation}
\item Like SIM1, we sample $\mathbf{x}_{0}^{(a)}$ from the uniform distribution on $S^{d-1}$ ($d=4$ in this case). Furthermore, we only keep those $\mathbf{x}_{0}^{(a)}$ with relatively large ICIS, namely $\wa(A, \mathbf{x}_{0}^{(a)}) > 0.2$. This ensures the practical identifiability of $A$ at $\mathbf{x}_{0}^{(a)}$.
\item Once $\mathbf{x}_{0}^{(a)}$ is sampled, we define
  \begin{equation}
    \label{eq:x0b}
    \mathbf{x}_{0}^{(b)} = \frac{(I_{4} - Q_{\cdot 4} Q_{\cdot 4}')\mathbf{x}_{0}^{(a)}}{\big| (I_{4} - Q_{\cdot 4} Q_{\cdot 4}')\mathbf{x}_{0}^{(a)} \big|}.
  \end{equation}
  By construction, $\mathbf{x}_{0}^{(b)}$ is a unit vector such that $\mathbf{x}_{0}^{(b)} \perp Q_{\cdot 4}$, so that ICIS equals zero in this case. According to our theoretical derivations, $A$ is not identifiable at $\mathbf{x}_{0}^{(b)}$ due to ill-positioned initial conditions.
\item We compute three sets of solution trajectories: case (A) corresponds with $(A, \mathbf{x}_{0}^{(a)})$, case (B) with $(A, \mathbf{x}_{0}^{(b)})$, and case (C) with $(B, \mathbf{x}_{0}^{(a)})$.
\end{enumerate}

For each case, we compute ICIS ($\wa$), SCN ($\tau$), PIS ($\wb$), and Stanhope's $\kappa$, based on both noisy and noise-free data. The results are illustrated in Figures~\ref{fig:SIM2-boxplots} and \ref{fig:SIM2-ROC}. We find that for noise-free data, SCN, PIS, and $\kappa$ perform very well, with almost perfect area under the curve (AUC) in receiver operating characteristic (ROC) analyses.  However, ICIS only has a relatively small AUC=0.773. This is not a surprise at all because ICIS is designed to detect un-identifiability issues associated with ill-positioned \emph{initial conditions} (case B), not un-identifiable systems that have \emph{repeated eigenvalues} (case C). This fact is also revealed in the corresponding boxplot in Figure~\ref{fig:SIM2-boxplots} (second column).

For data with noise, $\kappa$ is almost uninformative (AUC=0.503), but SCN, which is a smoothed extension of $\kappa$, works very well (AUC=0.946). It suggests that taking the smoothing effect into the consideration in SCN improves its utility as a classifier of identifiable systems.

While SCN has significantly better performance than $\kappa$ and ICIS, it is still an \emph{ad hoc} metric of practical identifiability that does not account for the uncertainty in $\hat{A}$ due to measurement error. In contrast, PIS is designed based on rigorous asymptotic analysis on the variance of $\hat{A}$, therefore PIS has the best performance (AUC=0.962).  That being said, we need to point out that from the computational perspective, SCN is more efficient and numerically robust, because SCN does not contain $(YSY')^{-1}$ and $(YSY')^{-2}$ terms used in PIS, which could have numerical issues if the dimension of the ODE system is large.  In summary, SCN could be considered as a simplified version of PIS that is less vulnerable to computational issues.

\begin{figure}
  \includegraphics[width=\textwidth]{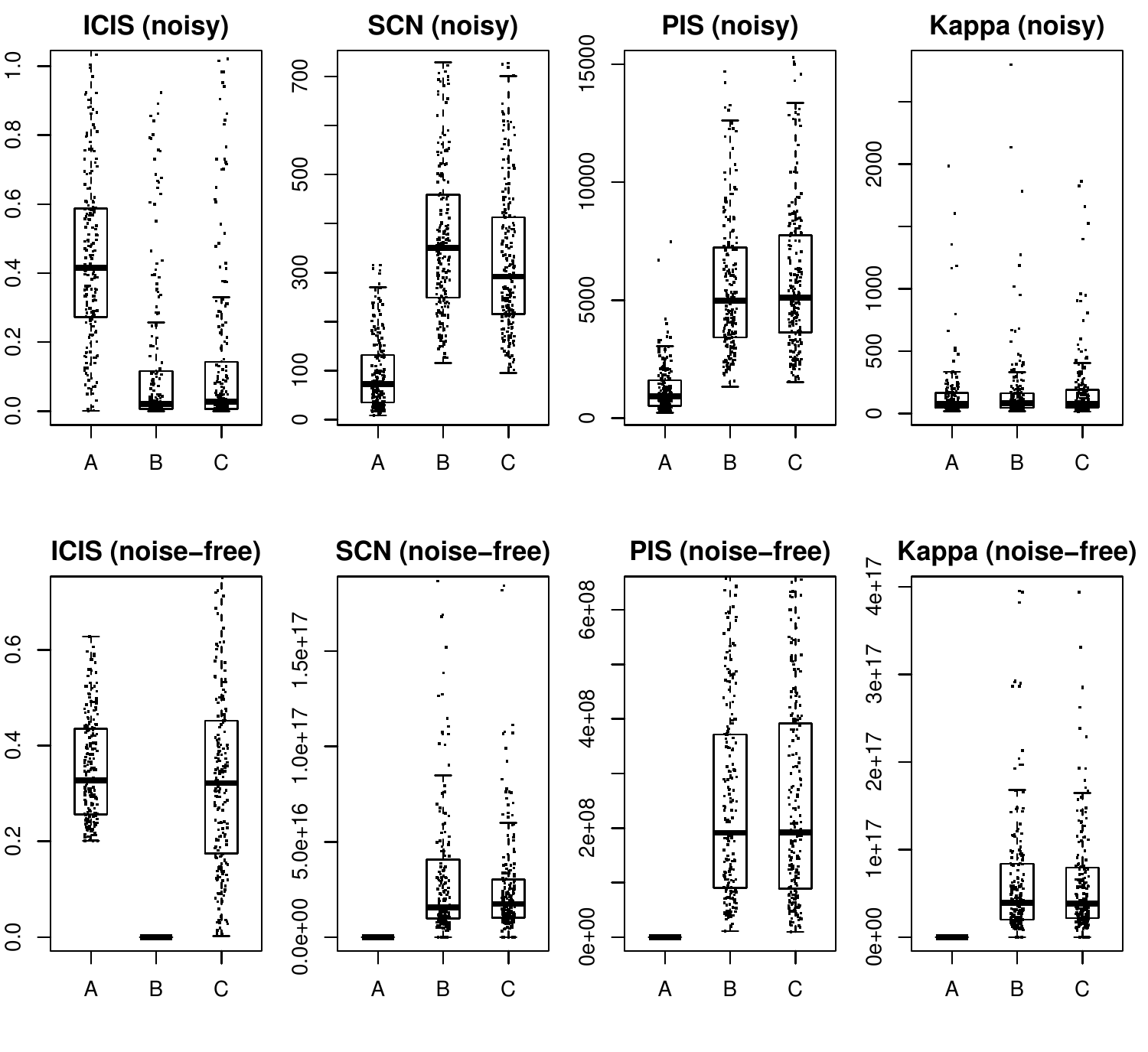}
  \caption{Values of four identifiability measures (ICIS, SCN, PIS, and Stanhope's $\kappa$) computed from 200 repetitions of SIM2. For all four scores, smaller values imply better identifiability. The first row of subfigures are computed from noisy data ($Y$), the second row of subfigures are computed from noise-free data ($X$). Each subfigure has three cases: (A) is the \textbf{identifiable} case generated by $(A,\mathbf{x}_{0}^{(a)})$; (B) is the \textbf{unidentifiable} case generated by $(A,\mathbf{x}_{0}^{(b)})$; (C) is the \textbf{unidentifiable} case generated by $(B,\mathbf{x}_{0}^{(b)})$. }
  \label{fig:SIM2-boxplots}
\end{figure}

\begin{figure}
  \includegraphics[width=\textwidth]{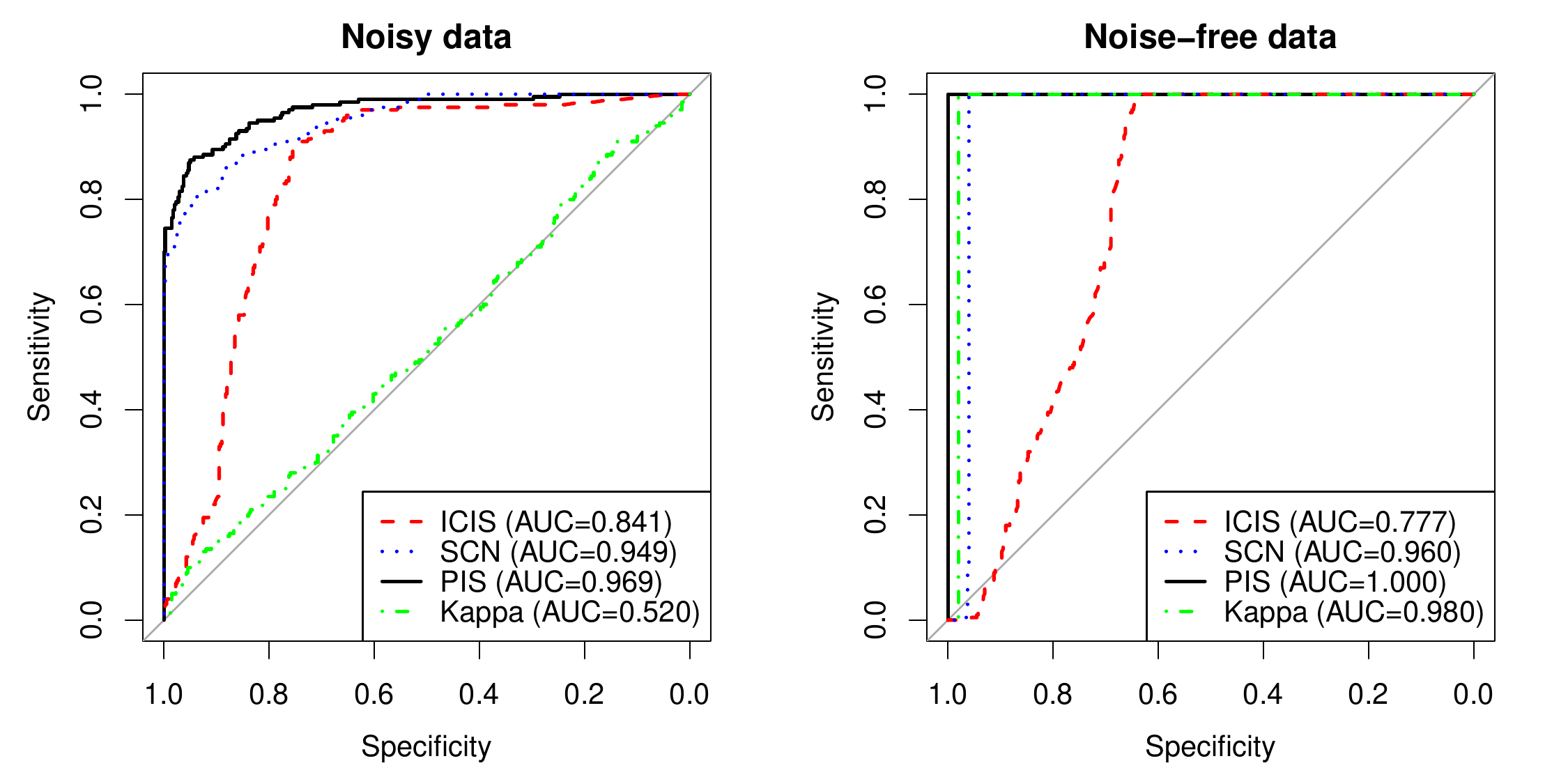}
  \caption{ROC curves of four identifiability measures (ICIS, SCN, PIS, and Stanhope's $\kappa$) in classifying case A (\textbf{identifiable} systems) from cases B and C (both are \textbf{unidentifiable} systems) in SIM2, with 200 repetitions. The left panel use data with noise; the right panel use noise-free data.}
  \label{fig:SIM2-ROC}
\end{figure}

Both noisy and noise-free data for all three cases were illustrated in Figure~\ref{fig:SIM2-fits}. Notably, visual examinations did not reveal apparent differences between the three cases, suggesting that the identifiability of the ODE system does not depend on obvious features in the solution trajectories.

\begin{figure}
  \includegraphics[width=\textwidth]{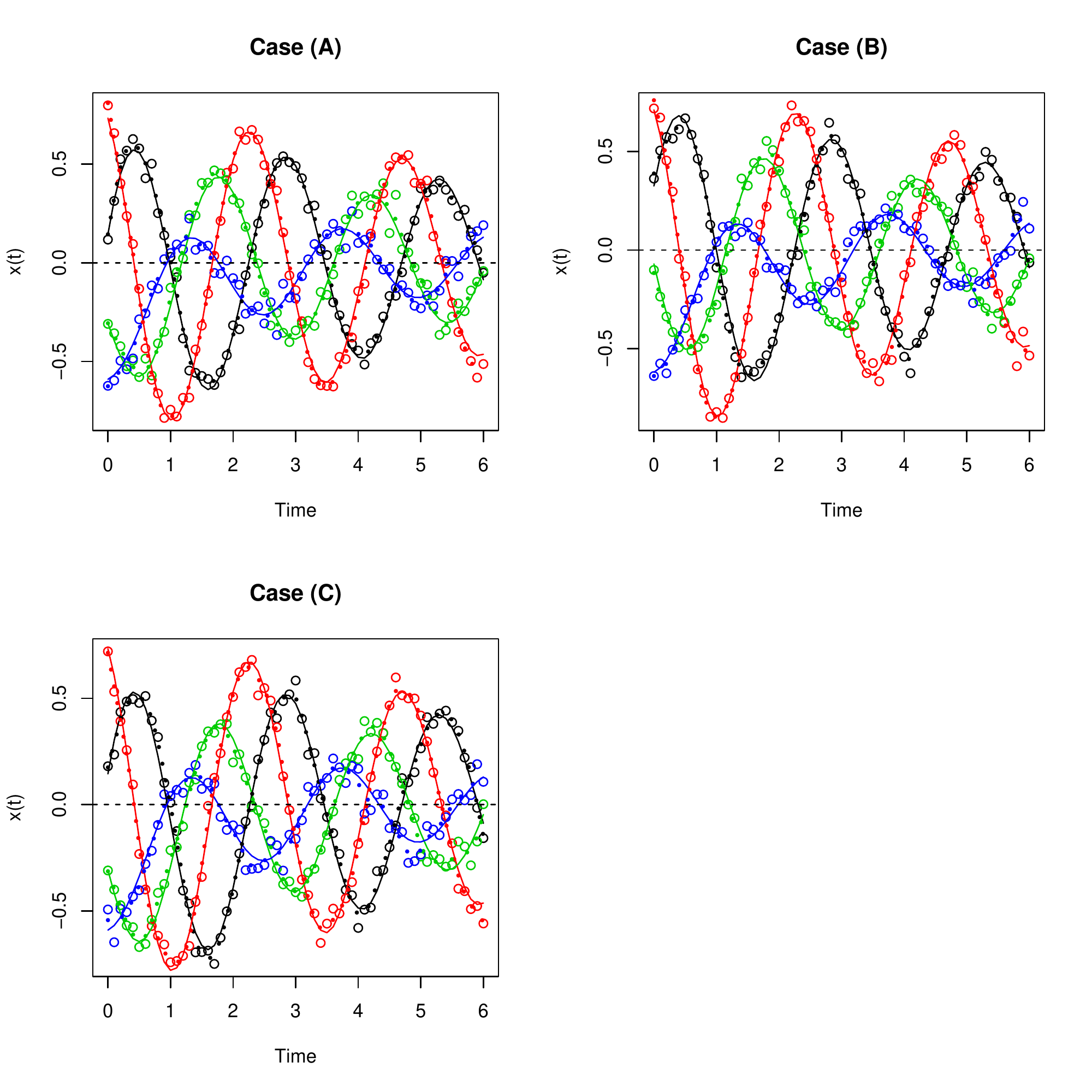}
  \caption{An illustration of the fitted solution curves in SIM2. Dots represent noisy data $y_{ij}$, four colors represent four dimensions. Solid curves are the noise-free solutions ($x(t)$) and dotted curves are the smoothed curves ($\hat{x}(t)$) fitted by roughness penalized splines. The same roughness penalty ($\lambda=0.001$) was used in all three cases.  Overall, the fitted curves agree with the noise-free data well.}
  \label{fig:SIM2-fits}
\end{figure}

\section{Conclusions}
\label{sec:conclusions}

Classical identifiability analyses for ODE systems typically depend on the availability of solution trajectories from arbitrarily many initial points. However, in many real world problems, the system matrix must be estimated from just one observed trajectory. In this case, identifiability depends not only on the properties of $A$, but also the initial condition $\mathbf{x}_{0}$. In this case, the $(A,\mathbf{x}_{0})$-identifiability used in our study is more appropriate than classical identifiability measures.

We develop an explicit formula of all matrices that are unidentifiable with $A$ at a given $\mathbf{x}_{0}$ in this study.  It enables researchers to gain better insight into identifiability analysis and help them design more practical simulation studies.

Another notable finding of our study is that when $A$ is coupled (not diagonal), an identifiability issue in just a one-dimensional invariant subspace could cause issues in many other elements of $A$ (e.g., Example~\ref{example:ex1-3d}).  Consequently, identifiability analyses that only depend on the \textbf{topology} of the network are insufficient in practice.

For high-dimensional cases, even if $\mathbf{x}_{0}$ is generated in a ``completely random'' fashion (e.g., $\mathbf{x}_{0} \sim N(0_{d}, I_{d})$), by chance, \textbf{one} invariant subspace of $A$ may have very little information, which in turn leads to practical identifiability issues.  In fact, we are able to prove that when $d\to \infty$, $\mathrm{ICIS}(A,\mathbf{x}_{0}) \to 0$ for a large class of random ODE systems, which suggests that the practical identifiability properties of low-dimensional and high-dimensional systems are fundamentally different. We believe it will be rewarding to derive more accurate convergence rates for $\mathrm{ICIS}$ as a function of $d$ in a future study. It will require  combining advanced techniques in random matrix theory, especially for ensembles of asymmetric matrices (e.g., Conjecture~\ref{thm:GinOE} in Supplementary Text) in which the $Q$ matrix are no longer orthogonal, with the identifiability analysis of ODE systems.

In this study, we also developed two scores, SCN and PIS, that use the entire dataset obtained at all time points, to quantify the practical identifiability for real world applications. Both SCN and PIS are more accurate than Stanhope's $\kappa$ when noise is present in the data, as shown by extensive simulation studies.

While our methods are developed for homogeneous systems, it should be relatively easy to generalize them for the following inhomogeneous linear ODE system
\begin{equation}
  \label{eq:inhomogeneous-ode}
  \begin{cases}
    D\mathbf{x}(t) = A\mathbf{x}(t) + b, \quad b \in \R^{d}, & t\in (0, T], \\
    \mathbf{x}(0) = \mathbf{x}_{0} \in \R^{d}.
  \end{cases}
\end{equation}

This is because Equation~\eqref{eq:inhomogeneous-ode} can be transformed into an \emph{equivalent homogeneous system} with a simple mathematical technique. Let $\mathbf{z}(t) = (\mathbf{x}(t), 1)' = (x_{1}(t), \dots, x_{d}(t), \, 1)' \in \R^{d+1}$. It satisfies the following ODE
\begin{equation}
  \label{eq:ODE-for-zt}
  \begin{cases}
    D\mathbf{z}(t) = \breve{A} \mathbf{z}(t), \qquad t\in (0, T], \\
    \mathbf{z}(0) = (\mathbf{x}_{0}, 1)' \in \R^{d+1}.
  \end{cases} \qquad \breve{A} :=
  \begin{pmatrix}
    A & b \\
    0_{d} & 0
  \end{pmatrix}.
\end{equation}

Therefore, the identifiability of Equation~\eqref{eq:inhomogeneous-ode} is the same as the identifiability of Equation~\eqref{eq:ODE-for-zt}, which is a homogeneous equation with the constraint that the last row of $\breve{A}$ must be zeros.  Let $M_{0} = \left\{B \in M_{(d+1)\times (d+1)}:\, M_{(d+1) \cdot} = 0_{d+1} \right\}$ be the set of $(d+1)\times (d+1)$-dimensional matrices such that their last rows equal $0_{d+1}$. The unidentifiability class associated with system $(A, b, \mathbf{x}_{0})$, denoted by $[A,b]_{\mathbf{x}_{0}}$, is the following subset of $M_{d\times d}$:
\begin{equation}
  \label{eq:inhomo-unident-class}
  [A, b]_{\mathbf{x}_{0}} = [\breve{A}]_{(\mathbf{x}_{0},1)'} \cap M_{0}.
\end{equation}

More future work is required to extend ICIS, SCN, and PIS for constrained systems, so that they can be used as practical guidance for applications with \textit{a priori} information.

In the near future, we plan to extend our work to the following family of nonlinear ODE system:
\begin{equation}
  \label{eq:linear-in-parameters}
  D\mathbf{x}(t) = A f\left( \mathbf{x}(t) \right), \qquad \mathbf{x}(0) = \mathbf{x}_{0}.
\end{equation}
Here $f(\cdot): \R^{d_{1}} \to \R^{d_{2}}$ is a \textbf{known} locally Lipschitz function of $\mathbf{x}(t)$, $A \in M_{d_{1}\times d_{2}}$ is the system matrix that needs to be estimated.  This system has been studied by Stanhope and colleagues, and their main conclusion (Theorem (5.3) in \cite{stanhope2014identifiability}) is very similar to that for the linear ODE systems: $A$ is identifiable at $\mathbf{x}_{0}$ if and only if the solution curve is not confine in a proper linear subspace of $\R^{d_{2}}$. To extend the SCN and PIS we developed in this study to Equation~\eqref{eq:linear-in-parameters}, we will need to study the sensitivity of an extended two-stage method that works for Equation~\eqref{eq:linear-in-parameters}.

Using linearization techniques, we believe SCN and PIS can be further extended to other types of nonlinear systems. To this end, we need: a) to approximate a nonlinear ODE system by a linear ODE at $\mathbf{x}(t)$; b) to propose a local version of the $(A,\mathbf{x}_{0})$-identifiability that works in a neighborhood of $A$ at $\mathbf{x}(t)$; c) to study the sensitivity of a reasonable parameter estimator for such system, and propose an identifiability score based on the useful information aggregated from all time points.

% \section{Tables and figures}

%%% Local Variables:
%%% mode: latex
%%% TeX-master: "ode_identifiability1_arxiv"
%%% End:

\bibliography{ident}

\begin{thebibliography}{10}
\expandafter\ifx\csname url\endcsname\relax
  \def\url#1{\texttt{#1}}\fi
\expandafter\ifx\csname urlprefix\endcsname\relax\def\urlprefix{URL }\fi
\expandafter\ifx\csname href\endcsname\relax
  \def\href#1#2{#2} \def\path#1{#1}\fi

\bibitem{butcher2014ordinary}
J.~Butcher, Ordinary differential equations, in: Walter Gautschi, Vol.~3,
  Springer, 2014, pp. 7--8.

\bibitem{commenges2011inference}
D.~Commenges, D.~Jolly, J.~Drylewicz, H.~Putter, R.~Thi{\'e}baut, Inference in
  {HIV} dynamics models via hierarchical likelihood, Computational Statistics
  \& Data Analysis 55~(1) (2011) 446--456.

\bibitem{DeJong2002}
H.~De~Jong, Modeling and simulation of genetic regulatory systems: a literature
  review, Journal of Computational Biology 9~(1) (2002) 67--103.

\bibitem{hemker1972numerical}
P.~W. Hemker, Numerical methods for differential equations in system simulation
  and in parameter estimation, Analysis and Simulation of Biochemical Systems
  28 (1972) 59--80.

\bibitem{holter2001dynamic}
N.~S. Holter, A.~Maritan, M.~Cieplak, N.~V. Fedoroff, J.~R. Banavar, Dynamic
  modeling of gene expression data, Proceedings of the National Academy of
  Sciences 98~(4) (2001) 1693--1698.

\bibitem{huang2006hierarchical}
Y.~Huang, D.~Liu, H.~Wu, Hierarchical bayesian methods for estimation of
  parameters in a longitudinal {HIV} dynamic system, Biometrics 62~(2) (2006)
  413--423.

\bibitem{lavielle2011maximum}
M.~Lavielle, A.~Samson, A.~Karina~Fermin, F.~Mentr{\'e}, Maximum likelihood
  estimation of long-term {HIV} dynamic models and antiviral response,
  Biometrics 67~(1) (2011) 250--259.

\bibitem{li2011large}
Z.~Li, P.~Li, A.~Krishnan, J.~Liu, Large-scale dynamic gene regulatory network
  inference combining differential equation models with local dynamic bayesian
  network analysis, Bioinformatics 27~(19) (2011) 2686--2691.

\bibitem{Lu2011}
T.~Lu, H.~Liang, H.~Li, H.~Wu, High-dimensional {ODEs} coupled with
  mixed-effects modeling techniques for dynamic gene regulatory network
  identification, Journal of the American Statistical Association 106~(496)
  (2011) 1242--1258.

\bibitem{Ramsay2007Parameter}
J.~O. Ramsay, G.~Hooker, D.~Campbell, J.~Cao, Parameter estimation for
  differential equations: a generalized smoothing approach (with discussion),
  Journal of the Royal Statistical Society 69~(5) (2007) 741--796.

\bibitem{moler2003nineteen}
C.~Moler, C.~Van~Loan, Nineteen dubious ways to compute the exponential of a
  matrix, twenty-five years later, SIAM review 45~(1) (2003) 3--49.

\bibitem{huang2006bayesian}
Y.~Huang, H.~Wu, A {Bayesian} approach for estimating antiviral efficacy in
  {HIV} dynamic models, Journal of Applied Statistics 33~(2) (2006) 155--174.

\bibitem{huang2010hierarchical}
Y.~Huang, H.~Wu, E.~P. Acosta, Hierarchical {Bayesian} inference for {HIV}
  dynamic differential equation models incorporating multiple treatment
  factors, Biometrical Journal 52~(4) (2010) 470--486.

\bibitem{li2005parameter}
Z.~Li, M.~R. Osborne, T.~Prvan, Parameter estimation of ordinary differential
  equations, IMA Journal of Numerical Analysis 25~(2) (2005) 264--285.

\bibitem{putter2002bayesian}
H.~Putter, S.~Heisterkamp, J.~Lange, F.~De~Wolf, A {Bayesian} approach to
  parameter estimation in {HIV} dynamical models, Statistics in Medicine
  21~(15) (2002) 2199--2214.

\bibitem{wu2019parameter}
L.~Wu, X.~Qiu, Y.-x. Yuan, H.~Wu, Parameter estimation and variable selection
  for big systems of linear ordinary differential equations: A matrix-based
  approach, Journal of the American Statistical Association 114~(526) (2019)
  657--667.

\bibitem{Xue2010}
H.~Xue, H.~Miao, H.~Wu, Sieve estimation of constant and time-varying
  coefficients in nonlinear ordinary differential equation models by
  considering both numerical error and measurement error, Annals of Statistics
  38~(4) (2010) 2351--2387.

\bibitem{miao2011identifiability}
H.~Miao, X.~Xia, A.~S. Perelson, H.~Wu, On identifiability of nonlinear ode
  models and applications in viral dynamics, SIAM review 53~(1) (2011) 3--39.

\bibitem{thowsen1978identifiability}
A.~THOWSEN, Identifiability of dynamic systems, International Journal of
  Systems Science 9~(7) (1978) 813--825.

\bibitem{stanhope2014identifiability}
S.~Stanhope, J.~Rubin, D.~Swigon, Identifiability of linear and
  linear-in-parameters dynamical systems from a single trajectory, SIAM Journal
  on Applied Dynamical Systems 13~(4) (2014) 1792--1815.

\bibitem{qiu2015diversity}
X.~Qiu, S.~Wu, S.~P. Hilchey, J.~Thakar, Z.-P. Liu, S.~L. Welle, A.~D. Henn,
  H.~Wu, M.~S. Zand, Diversity in compartmental dynamics of gene regulatory
  networks: the immune response in primary influenza a infection in mice, PloS
  one 10~(9) (2015).

\bibitem{sun2016controllability}
X.~Sun, F.~Hu, S.~Wu, X.~Qiu, P.~Linel, H.~Wu, Controllability and stability
  analysis of large transcriptomic dynamic systems for host response to
  influenza infection in human, Infectious Disease Modelling 1~(1) (2016)
  52--70.

\bibitem{wu2013high}
S.~Wu, Z.-P. Liu, X.~Qiu, H.~Wu, High-dimensional ordinary differential
  equation models for reconstructing genome-wide dynamic regulatory networks,
  in: Topics in applied statistics, Springer, New York, NY, 2013, pp. 173--190.

\bibitem{wu2014modeling}
S.~Wu, Z.-P. Liu, X.~Qiu, H.~Wu, Modeling genome-wide dynamic regulatory
  network in mouse lungs with influenza infection using high-dimensional
  ordinary differential equations, PloS one 9~(5) (2014).

\bibitem{mccullers2010influenza}
J.~A. McCullers, J.~L. McAuley, S.~Browall, A.~R. Iverson, K.~L. Boyd,
  B.~Henriques~Normark, Influenza enhances susceptibility to natural
  acquisition of and disease due to streptococcus pneumoniae in ferrets, The
  Journal of infectious diseases 202~(8) (2010) 1287--1295.

\bibitem{tunali1987new}
E.~Tunali, T.-J. Tarn, New results for identifiability of nonlinear systems,
  IEEE Transactions on Automatic Control 32~(2) (1987) 146--154.

\bibitem{jeffrey2005identifiability}
A.~M. Jeffrey, X.~Xia, I.~Craig, Identifiability of hiv/aids models,
  Deterministic and Stochastic models of AIDS epidemics and HIV infections with
  intervention (2005) 255--286.

\bibitem{xia2003identifiability}
X.~Xia, C.~H. Moog, Identifiability of nonlinear systems with application to
  hiv/aids models, IEEE transactions on automatic control 48~(2) (2003)
  330--336.

\bibitem{ginibre1965statistical}
J.~Ginibre, Statistical ensembles of complex, quaternion, and real matrices,
  Journal of Mathematical Physics 6 (1965) 440.

\bibitem{gohberg2006invariant}
I.~Gohberg, P.~Lancaster, L.~Rodman, Invariant subspaces of matrices with
  applications, SIAM, 2006.

\bibitem{lehmann1991eigenvalue}
N.~Lehmann, H.-J. Sommers, Eigenvalue statistics of random real matrices,
  Physical Review Letters 67~(8) (1991) 941--944.

\bibitem{tao2012topics}
T.~Tao, Topics in random matrix theory, Vol. 132, American Mathematical Society
  Providence, RI, 2012.

\bibitem{livan2018introduction}
G.~Livan, M.~Novaes, P.~Vivo, Introduction to random matrices: theory and
  practice, Vol.~26, Springer, 2018.

\bibitem{weyl1946classical}
H.~Weyl, The classical groups: their invariants and representations, Vol.~45,
  Princeton university press, 1946.

\bibitem{edelman1997probability}
A.~Edelman, The probability that a random real gaussian matrix haskreal
  eigenvalues, related distributions, and the circular law, Journal of
  Multivariate Analysis 60~(2) (1997) 203--232.

\bibitem{ferguson1996course}
T.~S. Ferguson, A course in large sample theory, Vol.~49, Chapman \& Hall
  London, 1996.

\end{thebibliography}

% now the supplementary text.
\pagebreak
% \widetext
\begin{center}
\textbf{\large Supplementary Text: Identifiability Analysis of Linear Ordinary Differential Equation Systems with a Single Trajectory}
\end{center}
\setcounter{equation}{0}
\setcounter{section}{0}
\setcounter{figure}{0}
\setcounter{table}{0}
\setcounter{page}{1}
\makeatletter
\renewcommand{\theequation}{S.\arabic{equation}}
\renewcommand{\thesection}{S\arabic{section}}
\renewcommand{\thetable}{S\arabic{table}}
\renewcommand{\thefigure}{S\arabic{figure}}
\renewcommand{\thetheorem}{S\arabic{theorem}}
\renewcommand{\theexample}{S\arabic{example}}

\section{Structural Identifiability is Unattainable for Linear ODE Systems}
\label{sec:about-struct-ident}

In this section, we prove the following statement:
\begin{proposition}\label{theo:unconditional-ident-unattainable}
  If the dimension $d$ is odd, there is no open subset $\Omega \subseteq M_{d\times d}$ on which ODE system~\eqref{eq:lin-ode} is unconditionally identifiable.  If the dimension $d$ is even, system~\eqref{eq:lin-ode} is not unconditionally identifiable for all $\Omega \subseteq M_{d\times d}$ such that $\lambda\left(M_{d\times d}\setminus \Omega\right) = 0$, where $\lambda(\cdot)$ is the Lebesgue measure on $M_{d\times d}$.
\end{proposition}

\begin{proof}
  First, we note that the structural identifiability defined in \ref{def:struct-ident} is a special case of the so-called \emph{unconditional identifiability} defined in Definition~2.4 \cite{stanhope2014identifiability} when $\Omega$ is set to be an \textbf{open and dense} subset of $M_{d\times d}$.

  According to Corollary~3.9 in \cite{stanhope2014identifiability}, ODE system~\eqref{eq:lin-ode} is unconditionally identifiable on an open set $\Omega \subset M_{d\times d}$ iff for every $A \in \Omega$, there is no left-eigenvector of $A$ that is orthogonal to every $\mathbf{x}_{0} \in \R^{d} \setminus \left\{ 0_{d} \right\}$. That immediately excludes matrices that has at least one real eigenvalue and eigenvector, which includes all cases when $d$ is odd.

  Now let us focus on the even-dimensional cases. Edelman showed in \cite{edelman1997probability} that for a random matrix $A \in M_{d\times d}$ with $i.i.d.$ normally distributed entries (the Ginibre ensemble), the probability of $A$ having a real eigenvalue is strictly greater than zero. Because the probability measure of the Ginibre ensemble and the Lebesgue measure on $M_{d\times d}$ are absolutely continuous with respect to each other, we know that we cannot find $\Omega$ such that: a) $A$ is unconditionally identifiable on $\Omega$, and b) $\lambda\left( M_{d\times d} \setminus \Omega\right) = 0$.
\end{proof}

\section{Practically Unidentifiable High-dimensional ODEs}
\label{sec:highdim-ode-unidentifiable}

In this section, we move Theorem~\ref{thm:GOE-like}, which states that if the dimension is high and system matrix $A$ is generated from a large class of random matrices, the ICIS converges to zero in $L^{2}$ (and in probability) when $d \to \infty$. 

First, we need to prove the following technical lemma.

\begin{lemma}
  \label{thm:uniform-sphere-extreme-Smin}
  Let $U = (U_{1}, U_{2}, \dots, U_{d})'$, $U \sim \mathrm{Unif}(S^{d-1})$ be a unit vector in $\R^{d}$ generated from the uniform distribution on $S^{d-1}$. Let $S_{\min} := \min_{i} U_{i}^{2}$.

  We have
  \begin{equation}
    \label{eq:asymp-unif-sphere-max}
    \frac{2d^{3}}{\pi}\cdot S_{\min} \stackrel{w}{\longrightarrow} \mathrm{Weibull}\Big(1, \frac{1}{2}\Big). \\
  \end{equation}
  Here $\mathrm{Weibull}\Big(1, \frac{1}{2}\Big)(x) = 1_{\R^{+}}(x) \cdot  \left(1-e^{-x^{1/2}}\right)$ is a Weibull distribution with scale parameter 1 and shape parameter 1/2.
\end{lemma}

\begin{proof}
  Since $U \sim \mathrm{Unif}(S^{d-1})$, there exists $Z \sim N(0_{d}, I_{d})$, such that $U = \frac{Z}{|Z|} = \left(\frac{Z_{1}}{|Z|}, \dots, \frac{Z_{d}}{|Z|} \right)'$.  Let $W_{\min} := \min_{i} Z_{i}^{2}$. Because $|Z|$ and $|Z|^{2}$ are a constant for all $i=1,2,\dots, d$, it is easy to see that $S_{\min} = \frac{W_{\min}}{|Z|^{2}}$.

  Based on the Fisher-Tippett-Gnedenko theorem~\cite{ferguson1996course} and notice that $W_{\min}$ has a lower bound ($W_{\min} \geqslant 0$), we can show that
  \begin{equation}
    \label{eq:asymp-Zmin-Zstar}
    \begin{gathered}
      \frac{2d^{2} W_{\min}}{\pi} \stackrel{w}{\longrightarrow} \mathrm{Weibull}\Big(1, \frac{1}{2} \Big).
    \end{gathered}
  \end{equation}

  In addition, we know that $|Z|^{2} := \sum_{i=1}^{d} Z_{i}^{2} \sim \chi^{2}_{d}$. Simple calculations show that when $d\to \infty$,
  \begin{equation}
    \label{eq:znorm-CLT}
    \begin{gathered}
      E \frac{|Z|^{2}}{d} = 1, \qquad \mathrm{var}\left( \frac{|Z|^{2}}{d} \right) = \frac{2}{d} \longrightarrow 0, \qquad \frac{|Z|^{2}}{d} \stackrel{L^{2}}{\longrightarrow} 1. \\
    \end{gathered}
  \end{equation}

  Using Slutsky's theorem, we have
  \begin{equation}
    \label{eq:Wmin-asymp}
    \begin{split}
      \frac{2d^{3} S_{\min}}{\pi} = \frac{ \frac{2d^{2} W_{\min}}{\pi} }{ \frac{|Z|^{2}}{d} } \stackrel{w}{\longrightarrow} \mathrm{Weibull}\Big(1, \frac{1}{2} \Big).
    \end{split}
  \end{equation}

\end{proof}

\begin{corollary}
  \label{thm:uniform-sphere-extreme-Smin-corollary}
  Let $\tilde{\mathbf{x}}_{0} \sim \mathrm{Unif}\left( r\cdot S^{d-1} \right)$ be a uniformly distributed random variable on a sphere in $\R^{d}$ with radius $d$, and $T \in O(d)$ be an arbitrary orthogonal matrix. Let
  \begin{equation*}
    S_{\min}(T,r) = \min_{i} \left(T\tilde{\mathbf{x}}_{0}\right)_{i}^{2}
  \end{equation*}
  be the smallest of the squared elements in vector $T\tilde{\mathbf{x}}_{0}$.

  We have
  \begin{equation}
    \label{eq:asymp-unif-sphere-max-withT}
    \frac{2d^{3}}{\pi r^{2}}\cdot S_{\min}(T,r) \stackrel{w}{\longrightarrow} \mathrm{Weibull}\Big(1, \frac{1}{2}\Big). \\
  \end{equation}

\end{corollary}

\begin{proof}
  It is easy to show that (a) $\left(T\tilde{\mathbf{x}}_{0}\right)_{i}^{2} = r^{2} \left(T \frac{\tilde{\mathbf{x}}_{0}}{r} \right)_{i}^{2}$, therefore $S_{\min}(T,r) = r^{2} S_{\min}(T,1)$, and (b) $\tilde{\mathbf{x}}_{0} \sim \mathrm{Unif}\left( r\cdot S^{d-1} \right)$ is invariant under orthogonal transformation, i.e., $T\tilde{\mathbf{x}}_{0} \stackrel{d}{=} \tilde{\mathbf{x}}_{0}$, therefore $S_{\min}(T,1) = S_{\min}$. In summary, $S_{\min}(T,r) \stackrel{d}{=} r^{2}S_{\min}$ for every $T \in O(d)$ and $r\in \R^{+}$, in which implies Equation~\eqref{eq:asymp-unif-sphere-max-withT}.
\end{proof}

We are now ready to prove Theorem~\ref{thm:GOE-like}.
\begin{proof}[Proof of Theorem~\ref{thm:GOE-like}]
As before, we write $A = Q\Lambda Q^{-1}$ as the Jordan canonical decomposition of $A$. Recall that ICIS is defined as $\wa(A, \mathbf{x}_{0}) := \min_{k} |w_{0,k}|$, where $w_{0,k} := Q^{-1} \mathbf{x}_{0}$. Here $Q$ is the basis system of all invariant subspaces of $A$. The symmetry assumption implies that: (a) $Q$ is orthogonal, so $Q^{-1} = Q'$, and (b) $\Lambda$ is diagonal. The orthogonal invariance assumption implies that the marginal distribution of $Q$ must be the standardized Haar measure on $O(d)$.

  Let $T\in O(d)$ be an arbitrary orthogonal matrix and let $\tilde{A} = TAT'$. Apparently, its orthogonal decomposition is $\tilde{A} = \tilde{Q} \Lambda \tilde{Q}'$, for $\tilde{Q} = TQ$.  Based on the orthogonal invariance of $A$, we know that $p(\tilde{Q}) = p(Q) = \mathrm{Haar}(O(d))$, and $\wa(\tilde{A}, \mathbf{x}_{0}) \stackrel{d}{=} \wa(A, \mathbf{x}_{0})$.

  Let us define $\varphi(\mathbf{x}) : \R^{d} \to \R^{+}$ be $\varphi(\mathbf{x}) := \max_{i} |Qx_{i}|$. Using this notation, $\wa(A, \mathbf{x}_{0}) = \varphi(Q^{-1}\mathbf{x}_{0})$, and
  \begin{equation}\label{w0star-ortho-invariant}
    \begin{gathered}
      \wa(\tilde{A}, \mathbf{x}_{0}) := \wa(TAT', \mathbf{x}_{0}) = \varphi(Q^{-1} T' \mathbf{x}_{0}) = \wa(A, T'\mathbf{x}_{0}). \\
      \wa(\tilde{A}, \mathbf{x}_{0}) \stackrel{d}{=} \wa(A, \mathbf{x}_{0}) \Longrightarrow \wa(A, T'\mathbf{x}_{0}) \stackrel{d}{=} \wa(A, \mathbf{x}_{0}).
    \end{gathered}
  \end{equation}

  The above equation implies that the distribution of $\wa(A, \mathbf{x}_{0})$ is statistically invariant under an arbitrary orthogonal transformation of $\mathbf{x}_{0}$, therefore its distribution depends only on $|\mathbf{x}_{0}|$. Note that the orbit of a fixed $\mathbf{x}_{0}$ under all orthogonal transformations is $|\mathbf{x}_{0}|\cdot S^{d-1}$, a sphere in $\R^{d}$ with radius $|\mathbf{x}_{0}|$. Furthermore, if $T \sim \mathrm{Haar}(O(d))$, the distribution of $T'\mathbf{x}_{0}$ must be the uniform distribution on $|\mathbf{x}_{0}|\cdot S^{d-1}$. In particular, if we let $T = Q'$, we have
  \begin{equation}
    \label{eq:w0star-ortho-invariant-to-diagonal}
    \wa(A, \mathbf{x}_{0}) \stackrel{d}{=} \wa(\Lambda, \tilde{\mathbf{x}}_{0}), \qquad \tilde{\mathbf{x}}_{0} := Q\mathbf{x}_{0} \sim \mathrm{Unif}\left(|\mathbf{x}_{0}|\cdot S^{d-1} \right).
  \end{equation}

  Conditionally on $|\mathbf{x}_{0}|=r$, $r\in \R^{+}$ and using Corollary~\ref{thm:uniform-sphere-extreme-Smin-corollary}, we know that
  \begin{equation}
    \label{eq:cond-exp-ICIS}
    \begin{split}
      E \left( \wa(A, \mathbf{x}_{0})^{2} \big| |\mathbf{x}_{0}|=r\right) &= E \left( \wa(\Lambda, \tilde{\mathbf{x}}_{0})^{2} \big| |\mathbf{x}_{0}|=r\right) \\
      &= E S_{\min}(T,r) \approx C_{\mathrm{Weibull}(1,1/2)} \cdot \frac{\pi r^{2}}{2d^{3}}.
    \end{split}
  \end{equation}
  Here $C_{\mathrm{Weibull}(1,1/2)}$ is the second order moment of $\mathrm{Weibull}(1,1/2)$.

  On the other hand, Assumption (b) states that as a function of $d$, $E |\mathbf{x}_{0}|^{2}$ grows at rate $o(d^{3})$. Therefore
  \begin{equation*}
    \begin{split}
      E \left( \wa(A, \mathbf{x}_{0})^{2} \right)  &= E_{r} \left(E \left( \wa(A, \mathbf{x}_{0})^{2} \big| |\mathbf{x}_{0}|=r\right) \right) \\
      &\approx C_{\mathrm{Weibull}(1,1/2)} \cdot \frac{\pi E |\mathbf{x}_{0}|^{2}}{2d^{3}} = o(1) \to 0.
    \end{split}
  \end{equation*}

  As a special case, if the second order moments of $x_{0,i}$ are bounded by a constant (as mentioned in Remark~\ref{remark:GOE-assumptions}), $E |\mathbf{x}_{0}|^{2}$ will grow at a speed of $O(d^{1})$, therefore $E \left( \wa(A, \mathbf{x}_{0})^{2} \right) = O(d^{-2})$.

\end{proof}

Below we propose a conjecture on asymmetric matrices.
\begin{conjecture}\label{thm:GinOE}
  We conjecture that Equation~\eqref{eq:w0star-converges-to-zero} is true for system matrix $A$ sampled from the Ginibre ensemble and initial condition $\mathbf{x}_{0}$ such that the second order moments of $x_{0,i}$ are bounded by a constant.
\end{conjecture}

Although we were not able to formally prove Conjecture~\ref{thm:GinOE} due to technical difficulties ($Q$ is no longer an orthogonal matrix for an asymmetric matrix sampled from GinOE), this conjecture was numerically verified by us in numerical experiments; for example, the mini-simulation described in Section~\ref{sec:ident-highdim}.

\section{Some Expected Values Related to Random Matrices}
\label{sec:exp-random-matrix}

Below we derive some expected values related to random matrices that are useful in deriving the PIS.

\begin{proposition}\label{thm:mean-random-matrix}
  Let $\bm{\epsilon} \in M_{d\times n}$ be a random matrix with $i.i.d.$ entries $\epsilon_{ij} \sim N(0, \sigma^{2})$. Let $A \in M_{n\times d}$ and $B \in M_{n\times n}$ be two deterministic matrices. We have
  \begin{equation}
    \label{eq:mean-random-matrix1}
    E\left( \bm{\epsilon} A \bm{\epsilon} \right) = \sigma^{2} A', \qquad E \left( \bm{\epsilon} B \bm{\epsilon}' \right) = \sigma^{2} \mathrm{tr}(B)\cdot I_{d}.
  \end{equation}
\end{proposition}

\begin{proof}
  Let $U = E\left( \bm{\epsilon} A \bm{\epsilon} \right)$ and $V = E \left( \bm{\epsilon} B \bm{\epsilon}' \right)$, we know that
  \begin{equation*}
    \begin{split}
      U_{ij} &= \sum_{l=1}^{n}\sum_{k=1}^{d} A_{lk} E (\epsilon_{il} \epsilon_{kj}) = A_{ji} E\epsilon_{ij}^{2} = \sigma^{2} A_{ji}. \\
      V_{ij} &= \sum_{l=1}^{n}\sum_{k=1}^{n} B_{lk} E (\epsilon_{il} \epsilon_{jk}) =
      \begin{cases}
        0, & i\ne j, \\
        \sum_{l=1}^{n} B_{ll} E\epsilon_{il}^{2} = \sigma^{2} \mathrm{tr}(B), & i=j.
      \end{cases}
    \end{split}
  \end{equation*}
\end{proof}

\section{Proof of Main Theorems}
\label{sec:mathematical-proofs}

% \section{Proof of Theorem }
% \label{sec:proof-theor-refthm:f}

Before we prove Theorem~\ref{thm:identifiability-based-on-pairwise-inprod}, we need to develop the following lemma that establishes the relationship between the linear independence of a set of functions and the invertibility of their pairwise inner product matrix.
\begin{lemma}\label{thm:func-indep-based-on-pairwise-inprod}
  Let $\mathbf{x}(t) \in \R^{d}$ be a set of \textbf{continuous} functions defined on $[0, T]$. $\mathbf{x}(t)$ is linearly independent iff its pairwise inner product matrix $\Sigma_{\mathbf{x}\mathbf{x}} := \Sigma_{\mathbf{x}\mathbf{x}}$ is invertible.
\end{lemma}
\begin{proof} \quad
  \begin{itemize}
  \item The ``$\Longleftarrow$'' part. Let us assume that member functions in $\mathbf{x}(t)$ are linearly dependent. There exist nonzero vector $\mathbf{c}$ such that $\mathbf{c}' \mathbf{x}(t) \equiv 0$ for all $t\in [0,T]$. Due to the bilinearity of the pairwise inner product matrix, we have
    \begin{equation*}
      \mathbf{c}' \Sigma_{\mathbf{x}\mathbf{x}} \mathbf{c} =
      \uinner{\mathbf{c}'\mathbf{x}(t)}{\mathbf{c}'\mathbf{x}(t)} = \uinner{0}{0} = 0.
    \end{equation*}
  \item The ``$\Longrightarrow$'' part. As we mentioned earlier, $\Sigma_{\mathbf{x}\mathbf{x}}$ is singular if and only if there exists nonzero $\mathbf{v}_{0} \in \R^{d}$ such that
    \begin{equation*}
      \mathbf{v}_{0}' \Sigma_{\mathbf{x}\mathbf{x}} \mathbf{v}_{0} = 0.
    \end{equation*}
    Due to bilinearity, we know $\uinner{\mathbf{v}_{0}'\mathbf{x}(t)}{\mathbf{v}_{0}'\mathbf{x}(t)} = \|\mathbf{v}_{0}'\mathbf{x}(t)\|_{L^{2}}^{2}= 0$. This implies that $\mathbf{x}(t)$ is essentially zero on $[0,T]$. Together with the \textbf{continuity} assumption of $\mathbf{x}(t)$, we know that $\mathbf{x}(t)$ must be zero for all $t \in [0,T]$.
  \end{itemize}
\end{proof}

\begin{corollary}\label{thm:lin-comb-functions-invertible-if-C-invertible}
  Assume that $\bm{\xi}(t) := (\xi_{1}(t), \xi_{2}(t), \dots, \xi_{d}(t))'$ is a set of linearly independent functions, $C \in M_{d\times d}$ is a matrix. Set of functions $\mathbf{x}(t) := C \bm{\xi}(t)$ is linearly independent iff matrix $C$ is invertible.
\end{corollary}
\begin{proof}
  Based on the bilinearity, we know that
  \begin{equation*}
    \Sigma_{\mathbf{x}\mathbf{x}} = C \Sigma_{\bm{\xi}} C'.
  \end{equation*}
  Lemma~\ref{thm:func-indep-based-on-pairwise-inprod} says that $\Sigma_{\bm{\xi}}$ is invertible.  So $\Sigma_{\mathbf{x}\mathbf{x}}$ is invertible (hence linearly independent) iff $C$ is invertible.
\end{proof}

The next lemma shows that when all eigenvalues are distinct (which is our assumption), the basic ``building blocks'' of $\mathbf{x}(t)$ are linearly independent. To this end, we first define $\mathbf{z}(t) = (z_{1}(t), z_{2}(t), \dots, z_{d}(t))'$ to be a vector of \emph{fundamental solutions} of Equation~\eqref{eq:lin-ode} in which $A$ has Jordan canonical form specified in Equation~\eqref{eq:jordan-decomp}. Elements in $\mathbf{z}(t)$ can be represented as follows
\begin{equation}\label{eq:fundamental-solutions}
  z_{i}(t) =
  \begin{cases}
    e^{c_{i}t}, & k=1, 2, \dots, K_{1}, \\
    e^{a_{k}}\cos b_{k}t, & k=K_{1}+\frac{i-K_{1}+1}{2},\; i=K_{1}+1, K_{1}+3, \dots, d-1. \\
    e^{a_{k}}\sin b_{k}t, & k=K_{1}+\frac{i-K_{1}}{2},\; i=K_{1}+2, K_{1}+4, \dots, d.
  \end{cases}
\end{equation}

As a reminder, we point out that the relationship between $i$ and $k$ is provided in Equation~\eqref{eq:i-k-functions}. In particular, the $k$th pair of complex eigenvalues $a_{k} \pm b_{k}i$, for $k=K_{1}+1, \dots, K$, corresponds with a pair of fundamental solutions $(z_{i(k)}(t), z_{i(k)+1}(t))$, where $i(k):=2k-K_{1}-1$.

\begin{lemma}\label{thm:building-blocks-indep}
  As a set of functions, $\mathbf{z}(t)$ is linearly independent iff all the eigenvalues of $A$ are distinct.
\end{lemma}

\begin{proof}
  It is easy to see that for real eigenvalues ($k=1,2,\dots, K_{1}$), we have
  \begin{equation*}
    D^{m}z_{k}(t) = c_{k}^{m} z_{k}(t).
  \end{equation*}

  For complex eigenvalues ($k=K_{1}+1, \dots, K$), let
  \begin{equation*}
    \begin{gathered}
      r_{k} := |a_{k}+b_{k}i| = \sqrt{a_{k}^{2} +b_{k}^{2}}, \qquad \theta_{k} := \arg(a_{k}+b_{k}i).\\
      a_{k,m} := r_{k}^{m}\cos(m\theta_{k}), \qquad b_{k,m}  := r_{k}^{m}\sin(m\theta_{k}).
    \end{gathered}
  \end{equation*}

  We have
  \begin{equation*}
    \begin{split}
      D^{m}
      \begin{pmatrix}
        z_{i(k)}(t) \\
        z_{i(k)+1}(t)
      \end{pmatrix} &=
      \begin{pmatrix}
        a_{k} & -b_{k} \\
        b_{k} & a_{k}
      \end{pmatrix}^{m}
      \begin{pmatrix}
        z_{i(k)}(t) \\
        z_{i(k)+1}(t)
      \end{pmatrix} \\
      &= \begin{pmatrix}
        a_{k,m} & -b_{k,m} \\
        b_{k,m} & a_{k,m}
      \end{pmatrix}
      \begin{pmatrix}
        z_{i(k)}(t) \\
        z_{i(k)+1}(t)
      \end{pmatrix} \\
      &=
      \begin{pmatrix}
        a_{k,m}z_{i(k)}(t) -b_{k,m}z_{i(k)+1}(t) \\
        b_{k,m}z_{i(k)}(t) +a_{k,m}z_{i(k)+1}(t)
      \end{pmatrix} \\
      &=
      \begin{pmatrix}
        z_{i(k)}(t) & -z_{i(k)+1}(t) \\
        z_{i(k)+1}(t) & z_{i(k)}(t)
      \end{pmatrix}
      \begin{pmatrix}
        a_{k,m} \\
        b_{k,m}
      \end{pmatrix}.
    \end{split}
  \end{equation*}

  Therefore, we can represent the $m$th derivative of $\mathbf{z}(t)$ as $D^{m}\mathbf{z}(t) = Z(t) C_{m}$, where
  \begin{equation*}
    Z(t) :=
    \begin{pmatrix}
      z_{1}(t) & & & & & \\
      & \ddots & & & & \\
      & & z_{K_{1}}(t) & & & \\
      & & &
      \begin{matrix}
        z_{K_{1}+1}(t) & -z_{K_{1}+2}(t) \\
        z_{K_{1}+2}(t) & z_{K_{1}+1}(t)
      \end{matrix} & & \\
      & & & & \ddots & \\
      & & & & &
      \begin{matrix}
        z_{d-1}(t) & -z_{d}(t) \\
        z_{d}(t) & z_{d-1}(t)
      \end{matrix}
    \end{pmatrix},
  \end{equation*}
  and $C_{m}$ is a vector with the following elements
  \begin{equation*}
    \quad C_{m} :=
    \begin{pmatrix}
      c_{1}^{m} \\
      \vdots \\
      c_{K_{1}}^{m} \\ \hline
      a_{K_{1}+1,m} \\
      b_{K_{1}+1,m} \\ \hline
      \vdots \\ \hline
      a_{K,m} \\
      b_{K,m} \\
    \end{pmatrix}.
  \end{equation*}
  As a special case, it is easy to see that $C_{0} = 1_{d}$. Let us define
  \begin{equation*}
    C :=
    \begin{pmatrix}
      C_{0} & C_{1} & \dots & C_{d}
    \end{pmatrix}.
  \end{equation*}

  The Wronskian of $\mathbf{z}(t)$ can be represented as
  \begin{equation}
    \label{eq:wronskian}
    \begin{split}
      W(t) &= \left| Z(t) C
      \right| = \left| Z(t) \right| \cdot \left| C \right|.
    \end{split}
  \end{equation}

  It is easy to show that
  \begin{equation}
    \label{eq:det-Zt}
    \begin{split}
      \left| Z(t) \right| &= \prod_{k=1}^{K_{1}} z_{k}(t) \cdot \prod_{k=K_{1}+1}^{K} \left( z_{k}^{2} +z_{k+1}^{2}\right) \\
      &= \exp\left( \Big(\sum_{k=1}^{K_{1}} c_{k} +\sum_{k=K_{1}+1}^{K} a_{k} \Big)t \right) \ne 0.
    \end{split}
  \end{equation}

  Therefore, $W(t) \ne 0$, hence elements in $\mathbf{z}(t)$ are linearly independent if and only if $|C| \ne 0$.

  The determinant of $C$ can be computed by the following technique. First, we notice that
  \begin{equation*}
    \begin{gathered}
      \begin{pmatrix}
        1 & i \\
        1 & -i
      \end{pmatrix}
      \begin{pmatrix}
        a_{k,m} \\
        b_{k,m}
      \end{pmatrix} =
      \begin{pmatrix}
        a_{k,m} +b_{k,m}i \\
        a_{k,m} -b_{k,m}i
      \end{pmatrix}
      =
      \begin{pmatrix}
        \lambda_{i(k)}^{m} \\
        \lambda_{i(k)+1}^{m}
      \end{pmatrix}, \qquad k=K_{1}+1, \dots, K. \\
    \end{gathered}
  \end{equation*}
  Here $(\lambda_{i(k)}, \lambda_{i(k)+1})$ are the pair of complex eigenvalues associated with the $k$th block.  Because $ \left(\begin{smallmatrix}
        1 & i \\
        1 & -i
      \end{smallmatrix}\right)^{-1} =
    \left(\begin{smallmatrix}
        \frac{1}{2} & \frac{1}{2} \\
        \frac{-i}{2} & \frac{i}{2}
      \end{smallmatrix}\right)
$, we can rewrite the $C$ matrix as follows
  \begin{equation}\label{eq:C-vandermonde}
    \begin{split}
      C &=
      \begin{pmatrix}
        I_{K_{1}} & & & \\
        &
        \begin{matrix}
          \frac{1}{2} & \frac{1}{2} \\
          \frac{-i}{2} & \frac{i}{2}
        \end{matrix} & & \\
        & & \ddots & \\
        & & & \begin{matrix}
          \frac{1}{2} & \frac{1}{2} \\
          \frac{-i}{2} & \frac{i}{2}
        \end{matrix} \\
      \end{pmatrix}
      \begin{pmatrix}
        1 & \lambda_{1} & \lambda_{1}^{2} & \dots & \lambda_{1}^{d-1} \\
        1 & \lambda_{2} & \lambda_{2}^{2} & \dots & \lambda_{2}^{d-1} \\
        \vdots & \vdots & \vdots & \ddots & \vdots \\
        1 & \lambda_{d} & \lambda_{d}^{2} & \dots & \lambda_{d}^{d-1} \\
      \end{pmatrix}
    \end{split}
  \end{equation}

  Therefore
  \begin{equation}
    \label{eq:det-C}
      \left|C\right| = \left( \frac{i}{2} \right)^{K_{2}} \prod_{1\leqslant i < j \leqslant d} (\lambda_{j}-\lambda_{i}), \qquad \left|C\right| \ne 0 \text{ iff all $\lambda_{i}$ are distinct.} 
  \end{equation}

\end{proof}

\begin{proof}[Proof of Theorem~\ref{thm:identifiability-based-on-pairwise-inprod}]
  Again, I first prove the special case in which all eigenvalues are real. In this case, we know that
  \begin{equation}
    \label{eq:proof-fullrank-pairwise-inprod-all-real}
    \begin{split}
      \mathbf{x}(t) &= e^{At}\mathbf{x}_{0} = Q\cdot \mathrm{diag}(e^{c_{i}t}) \cdot Q^{-1}\mathbf{x}_{0} \\
      &= Q
      \begin{pmatrix}
        w_{0,1}e^{c_{1}t} \\
        w_{0,2}e^{c_{2}t} \\
        \vdots \\
        w_{0,d}e^{c_{d}t} \\
      \end{pmatrix} = Q \cdot \mathrm{diag}(w_{0,i})
      \begin{pmatrix}
        e^{c_{1}t} \\
        \vdots \\
        e^{c_{d}t}
      \end{pmatrix}.
    \end{split}
  \end{equation}
  By Lemma~\ref{thm:building-blocks-indep}, $(e^{c_{1}t}, \dots, e^{c_{d}t})'$ is linearly independent, so $\mathbf{x}(t)$ is linearly independent iff $Q\cdot \mathrm{diag}(w_{0,i})$ is invertible.  Now $Q$ is invertible based on our assumption, according to Corollary~\ref{thm:lin-comb-functions-invertible-if-C-invertible}, whether or not $\mathbf{x}(t)$ is linearly independent only depends on the invertibility of $\mathrm{diag}(w_{0,i})$, which in turn is equivalent to $w_{0,i} \ne 0$ for $i=1,2,\dots, d$. Using Theorem~\ref{thm:identifiability}, we know that $A$ is identifiable at $\mathbf{x}_{0}$ iff all $|w_{0,k}| \ne 0$, which is equivalent to the invertibility of matrices $\mathrm{diag}(w_{0,i})$ and $Q\cdot \mathrm{diag}(w_{0,i})$ (because $Q$ is invertible), which in turn is equivalent to the linear independence of $(e^{c_{1}t}, \dots, e^{c_{d}t})'$ and the invertibility of $\Sigma_{\mathbf{x}\mathbf{x}}$.
\end{proof}

\begin{proof}[Proof of Theorem~\ref{thm:functional-two-stage-perfect}]
  The ``$\Longrightarrow$'' direction. By assumption, $D\mathbf{x}(t) = A \mathbf{x}(t)$, we have
  \begin{equation*}
    \Sigma_{D\mathbf{x}, \mathbf{x}} = \uinner{A \mathbf{x}(t)}{\mathbf{x}(t)} = A \Sigma_{\mathbf{x}\mathbf{x}}.
  \end{equation*}
  Because $\Sigma_{\mathbf{x}\mathbf{x}}$ is invertible, we can multiply both sides of Equation~\eqref{eq:two-stage-equation} by $\Sigma_{\mathbf{x}\mathbf{x}}^{-1}$, therefore   $A$ must be $\Sigma_{D\mathbf{x}, \mathbf{x}} \Sigma_{\mathbf{x}\mathbf{x}}^{-1}$.

  The ``$\Longleftarrow$'' direction. Because $\Sigma_{\mathbf{x}\mathbf{x}}$ is invertible, based on Theorem~\ref{thm:identifiability-based-on-pairwise-inprod}, $A$ is identifiable at $\mathbf{x}_{0}$, which means that there is only one $A$ that can produce the solution curve $\mathbf{x}(t)$, therefore $\Sigma_{D\mathbf{x}, \mathbf{x}} \Sigma_{\mathbf{x}\mathbf{x}}^{-1}$ must be the system matrix.
\end{proof}

\begin{proof}[Proof of Theorem~\ref{thm:functional-two-stage-continuous}]

  Due to Theorem~\ref{thm:functional-two-stage-perfect}, we know that $A = \Sigma_{D\mathbf{x}, \mathbf{x}} \cdot \Sigma_{\mathbf{x}\mathbf{x}}^{-1}$. Therefore, it suffices to show that $\hat{A}$ is a continuous matrix-valued functional of $\hat{\mathbf{x}}(t)$ and $D\hat{\mathbf{x}}(t)$.  Clearly, we have
  \begin{equation}
    \label{eq:SigmaXX-cont}
    \begin{split}
      \uinner{\hat{\mathbf{x}}_{i}}{\hat{\mathbf{x}}_{i'}} -\uinner{\mathbf{x}_{i}}{\mathbf{x}_{i'}} &= \uinner{\hat{\mathbf{x}}_{i} -\mathbf{x}_{i}}{\hat{\mathbf{x}}_{i'}} +\uinner{\mathbf{x}_{i}}{\hat{\mathbf{x}}_{i'}-\mathbf{x}_{i'}}, \\
      \left(\uinner{\hat{\mathbf{x}}_{i}}{\hat{\mathbf{x}}_{i'}} -\uinner{\mathbf{x}_{i}}{\mathbf{x}_{i'}}\right)^{2} &\leqslant (\uinner{\hat{\mathbf{x}}_{i} -\mathbf{x}_{i}}{\hat{\mathbf{x}}_{i'}})^{2} +(\uinner{\mathbf{x}_{i}}{\hat{\mathbf{x}}_{i'}-\mathbf{x}_{i'}})^{2} \\
      &\quad +2|\uinner{\hat{\mathbf{x}}_{i} -\mathbf{x}_{i}}{\hat{\mathbf{x}}_{i'}}|\cdot |\uinner{\mathbf{x}_{i}}{\hat{\mathbf{x}}_{i'}-\mathbf{x}_{i'}}| \\
      &\leqslant \delta_{1}^{2} \left( \|\mathbf{x}_{i}\| +\|\mathbf{x}_{i'}\|\right)^{2}. \\
      \|\Sigma_{\hat{\mathbf{x}}\hat{\mathbf{x}}} -\Sigma_{\mathbf{x}\mathbf{x}}\|_{F}^{2} &= \sum_{i,i'=1}^{d} \left(\uinner{\hat{\mathbf{x}}_{i}}{\hat{\mathbf{x}}_{i'}} -\uinner{\mathbf{x}_{i}}{\mathbf{x}_{i'}}\right)^{2} dt \\
      &\leqslant  2\delta_{1}^{2} \|\mathbf{x}\|^{2}.
    \end{split}
  \end{equation}

  \begin{equation}
    \label{eq:SigmaDxx-cont}
    \begin{split}
      \uinner{D\hat{\mathbf{x}}}{\hat{\mathbf{x}}} -\uinner{D\hat{\mathbf{x}}}{\hat{\mathbf{x}}}  &= \uinner{D\hat{\mathbf{x}}_{i} -D\mathbf{x}_{i}}{\hat{\mathbf{x}}_{i'}} +\uinner{D\mathbf{x}_{i}}{\hat{\mathbf{x}}_{i'}-\mathbf{x}_{i'}}, \\
      \left(\uinner{D\hat{\mathbf{x}}}{\hat{\mathbf{x}}} -\uinner{D\hat{\mathbf{x}}}{\hat{\mathbf{x}}}\right)^{2} &\leqslant \left(\delta_{2}\|\hat{\mathbf{x}}_{i'}\| +\delta_{1}\|D\mathbf{x}_{i}\|\right)^{2} \\
      &\leqslant \left( \delta_{2}\|\mathbf{x}_{i'}\| +\delta_{1} +\delta_{1}\|D\mathbf{x}_{i}\| \right) \leqslant \delta\left( \|\mathbf{x}_{i'}\| +\|D\mathbf{x}_{i}\| +1 \right). \\
      \|\Sigma_{D\hat{\mathbf{x}}, \hat{\mathbf{x}}} -\Sigma_{D\mathbf{x}, \mathbf{x}}\|_{F}^{2} &= \sum_{i,i'=1}^{d} \left(\uinner{D\hat{\mathbf{x}}_{i}}{\hat{\mathbf{x}}_{i'}} -\uinner{D\mathbf{x}_{i}}{\mathbf{x}_{i'}}\right)^{2} dt \\
      &\leqslant  \delta^{2}\left( \|\mathbf{x}\|^{2} +\|D\mathbf{x}\|^{2} +1\right).
    \end{split}
  \end{equation}

  In other words, as matrix-valued functionals of $D\mathbf{x}$ and $\mathbf{x}$, $\Sigma_{\hat{\mathbf{x}}, \hat{\mathbf{x}}}$ and $\Sigma_{D\hat{\mathbf{x}}, \hat{\mathbf{x}}}$ are continuous w.r.t. the $L^{2}$ metric. Based on Theorem~\ref{thm:identifiability-based-on-pairwise-inprod}, the assumption that the ODE system is identifiable at $\mathbf{x}_{0}$ implies that $\Sigma_{\mathbf{x}\mathbf{x}}$ is invertible. Consequently, $\hat{A} := \hat{\Sigma}_{D\mathbf{x}, \mathbf{x}} \cdot  \hat{\Sigma}_{\mathbf{x}\mathbf{x}}^{-1}$ must be a matrix-valued continuous functional of $D\mathbf{x}$ and $\mathbf{x}$.
\end{proof}

\section{Identifiability Issues Induced by Repeated Eigenvalues}
\label{sec:repeated-eigenvalues}

In this subsection, we argue that when there exist repeated eigenvalues (real or complex), $A$ will not be identifiable for any $\mathbf{x}_{0}$. We then provide a representation of $[A]_{\mathbf{x}_{0}}$ in this case.  Although it is well known that the set of matrices with repeated eigenvalues has Lebesgue measure zero in $M_{d\times d}$, In reality, $A$ may still have eigenvalues with similar numerical values due to randomness, hence work presented in this section is relevant for real world applications with uncertainty.

First, we would like to provide an interpretation based on the Jordan canonical decomposition.  Suppose there exists a real eigenvalue of $A$ with multiplicity greater than one. WLOG, we can always rearrange the JCP so that these repeated eigenvalues are ordered as the first $m_{1}$ eigenvalues, $\lambda_{1} = \dots = \lambda_{m_{1}}$. The combined Jordan block for these repeated eigenvalues is
\begin{equation*}
  \Lambda_{1} =
  \begin{pmatrix}
    \lambda_{1} & & \\
    & \ddots & \\
    & & \lambda_{1}
  \end{pmatrix} = I_{m_{1}} \otimes J_{1}, \qquad J_{1} := (\lambda_{1})_{1\times 1}.
\end{equation*}

Similarly, if $A$ has $m_{1}$ pairs of repeated complex eigenvalues, we can arrange the JCP so that they are the first $2m_{1}$ eigenvalues (denoted as $a_{1}\pm b_{1}i$) in the semi-diagonal matrix $\Lambda$, thus the combined Jordan block of them is the following $2m_{k}\times 2m_{k}$-dimensional semi-diagonal matrix
\begin{equation*}
  \Lambda_{1} = I_{m_{1}}\otimes J_{1}, \qquad J_{1} :=
  \begin{pmatrix}
    a_{1} & -b_{1} \\
    b_{1} & a_{1}
  \end{pmatrix}.
\end{equation*}

In either case, the combined Jordan block of the repeated eigenvalue, $\Lambda_{1} = I_{m_{1}}\otimes J_{1}$, is the top-left block of the middle matrix in the JCP of $A$:
\begin{equation}\label{eq:JCP-repeated-eigen}
  \begin{gathered}
    A =
    \begin{pmatrix}
      Q_{1} & Q_{-1}
    \end{pmatrix}
    \begin{pmatrix}
      \Lambda_{1} & \\
      & \Lambda_{-1}
    \end{pmatrix}
    \begin{pmatrix}
      R_{1} \\
      R_{-1}
    \end{pmatrix} = Q_{1} \Lambda_{1} R_{1} + Q_{-1} \Lambda_{-1} R_{-1}.
  \end{gathered}
\end{equation}

Here $\Lambda_{-1}$ represents eigenvalues not in $\Lambda_{1}$, $Q_{1}$ is the collection of the first $m_{1}$ (or $2m_{1}$, if the repeated eigenvalues are complex) columns of $Q$, which is a basis for the invariant subspace associated with $\Lambda_{1}$. Likewise, $Q_{-1}$ is the collection of remaining column vectors of $Q$ and a basis of the invariant subspace associated with $\Lambda_{-1}$.  $R_{1}$ and $R_{-1}$ are the corresponding left- and right-submatrices of $Q^{-1}$.

When considered as a $m_{1}\times m_{1}$-dimensional (if the repeated eigenvalue is real) or $2m_{1}\times 2m_{1}$-dimensional (if the repeated eigenvalue is complex) \emph{subsystem}, $\Lambda_{1}$ is not identifiable at any initial point $\mathbf{v} \in \R^{m_{1}}$ and the solution curves satisfies the following equation
\begin{equation}
  \label{eq:matrix-exponential-repeated}
  e^{tA} \mathbf{v} = e^{t(A + UDU')} \mathbf{v}.
\end{equation}
Here $D$ is an arbitrary matrix and $U$ is a semi-orthogonal matrix that depends on the initial condition $\mathbf{v}$. Technical details of these results are summarized in Lemmas~\ref{thm:matrix-exponential-repeated-real} and \ref{thm:matrix-exponential-repeated-complex} below.

\begin{lemma}\label{thm:matrix-exponential-repeated-real}
  Let $\lambda_{1}\in \R$ and  $A = \lambda_{1}I_{m_{1}}$, $\mathbf{v} \in \R^{m_{1}}$.  We have
  \begin{equation*}
    e^{tA} \mathbf{v} = e^{t(A + UDU')} \mathbf{v}.
  \end{equation*}
  Here $U \in M_{m_{1}\times (m_{1}-1)}$ is a semi-orthogonal matrix such that $U'\mathbf{v} = \mathbf{0}$, and $D$ is an arbitrary $(m_{1}-1)\times (m_{1}-1)$-dimensional matrix.
\end{lemma}
\begin{proof}
  WLOG, we may assume that $\|\mathbf{v}\|=1$. By construction, matrix $W :=
  \begin{pmatrix}
    \mathbf{v} & U
  \end{pmatrix}
  $ is orthogonal.  It is easy to see that
  \begin{equation*}
    \begin{split}
      A + UDU' &= W \left( \lambda_{1} I_{m_{1}} \right) W' +
      \begin{pmatrix}
        \mathbf{v} & U
      \end{pmatrix}
      \begin{pmatrix}
        0 & \mathbf{0}\\
        \mathbf{0} & D
      \end{pmatrix}
      \begin{pmatrix}
        \mathbf{v}' \\
        U'
      \end{pmatrix} \\
      &= W \tilde{D} W', \qquad \tilde{D} :=
      \begin{pmatrix}
        \lambda_{1} & \\
        & \lambda_{1}I_{m_{1}-1} + D
      \end{pmatrix}.
    \end{split}
  \end{equation*}

  \begin{equation*}
    \begin{gathered}
      e^{tA} = e^{t\lambda_{1}} I_{m_{1}}, \qquad e^{tA} \mathbf{v} = e^{t\lambda_{1}} \cdot \mathbf{v}. \\
      \begin{split}
        e^{t(A + UDU')}\mathbf{v} &= W e^{t \tilde{D}} W'\mathbf{v} =
        \begin{pmatrix}
          \mathbf{v} & U
        \end{pmatrix}
        \begin{pmatrix}
          e^{t\lambda_{1}} & \\
          & e^{t(\lambda_{1}I_{m_{1}-1} + D)}
        \end{pmatrix}
        \begin{pmatrix}
          1 \\
          \mathbf{0}
        \end{pmatrix} \\
        &= e^{t\lambda_{1}}\cdot \mathbf{v} = e^{tA}\mathbf{v}.
      \end{split}
    \end{gathered}
  \end{equation*}
\end{proof}

Likewise, for repeated complex eigenvalues, we have the following lemma.
\begin{lemma}\label{thm:matrix-exponential-repeated-complex}
  Let $a \pm bi \in \mathbb{C}$ and $A = I_{m_{1}}\otimes
  \left(
    \begin{smallmatrix}
      a & -b \\
      b & a
    \end{smallmatrix}
  \right)
  $.  For $\mathbf{v} \in \R^{2m_{1}}$, we have
  \begin{equation*}
    e^{tA} \mathbf{v} = e^{t(A + UDU')} \mathbf{v}.
  \end{equation*}
  Here $D \in M_{(2m_{1}-2)\times (2m_{1}-2)}$ is an arbitrary matrix and $U \in M_{2m_{1}\times (2m_{1}-2)}$ is a semi-orthogonal matrix such that
  \begin{equation}
    \label{eq:complex-U-condition}
    U'
    \begin{pmatrix}
      \mathbf{v} & \bar{\mathbf{v}}
    \end{pmatrix} = \mathbf{0}_{2m_{1}}, \qquad \bar{\mathbf{v}} := \left( I_{m_{1}} \otimes
      \left(
        \begin{smallmatrix}
          0 & -1 \\
          1 & 0
        \end{smallmatrix}
      \right) \right) \mathbf{v}.
  \end{equation}
\end{lemma}

\begin{proof}
  WLOG, assume that $\|\mathbf{v}\|=1$.
  Let $W :=
  \begin{pmatrix}
    \mathbf{v} & \bar{\mathbf{v}} & U
  \end{pmatrix}
  $, it is easy to check that $W$ is an orthogonal matrix.  For simplicity, let us denote matrix $\left(
    \begin{smallmatrix}
      a & -b \\
      b & a
    \end{smallmatrix}
  \right)$ by $\Lambda_{2}$ and matrix
  $
  \begin{pmatrix}
    \mathbf{v} & \bar{\mathbf{v}}
  \end{pmatrix}
  $ by $V_{2}$.  Based on matrix analysis, we can derive
  \begin{equation}
    \label{eq:W-commutes-with-complex-diag}
    \begin{gathered}
      \left( I_{m_{1}} \otimes
        \Lambda_{2}
      \right)
      V_{2}
      =
      V_{2}\Lambda_{2}, \\
      W'
      \left( I_{m_{1}} \otimes
        \Lambda_{2}
      \right) W =
      I_{m_{1}} \otimes \Lambda_{2}.
    \end{gathered}
  \end{equation}
  Therefore
  \begin{equation*}
    \begin{split}
      A + UDU' &= W \left( I_{m_{1}} \otimes
        \Lambda_{2}
      \right) W' +
      \begin{pmatrix}
        V_{2} & U
      \end{pmatrix}
      \begin{pmatrix}
        \mathbf{0}_{2\times 2} & \mathbf{0}_{2\times (2m_{1}-2)}\\
        \mathbf{0}_{(2m_{1}-2)\times 2} & D
      \end{pmatrix}
      \begin{pmatrix}
        V_{2}' \\
        U'
      \end{pmatrix} \\
      &= W \tilde{D} W', \qquad \tilde{D} :=
      \begin{pmatrix}
        \Lambda_{1} & \\
        & I_{m_{1}-1}\otimes \Lambda_{2} + D
      \end{pmatrix}.
    \end{split}
  \end{equation*}

  \begin{equation*}
    \begin{gathered}
      e^{tA} = I_{m_{1}}\otimes e^{t\Lambda_{1}}, \qquad e^{tA} V_{2} = V_{2} \cdot e^{t\Lambda_{1}}. \\
      \begin{split}
        e^{t(A + UDU')}V_{2} &= W e^{t \tilde{D}} W' V_{2} =
        \begin{pmatrix}
          V_{2} & U
        \end{pmatrix}
        \begin{pmatrix}
          e^{t\Lambda_{1}} & \\
          & e^{t(I_{m_{1}-1}\otimes \Lambda_{2} + D)}
        \end{pmatrix}
        \begin{pmatrix}
          I_{2} \\
          \mathbf{0}
        \end{pmatrix} \\
        &= e^{t\Lambda_{1}}\cdot V_{2} = e^{tA} V_{2}.
      \end{split}
    \end{gathered}
  \end{equation*}
  The last equality implies that $e^{t(A + UDU')} \mathbf{v} = e^{tA}\mathbf{v}$ because $\mathbf{v}$ is the first column of $V_{2}$.
\end{proof}

As a direct consequence of Lemmas~\ref{thm:matrix-exponential-repeated-real} and \ref{thm:matrix-exponential-repeated-complex}, we have the following theorem.

\begin{theorem}\label{thm:repeated-eigen}
  Let $A$ be a matrix with repeated real or complex eigenvalues but no Jordan block with nilpotent components. WLOG, we arrange the JCF such that $\lambda$, the repeated real eigenvalue or pair of complex eigenvalues, corresponds with the first Jordan block:
  \begin{equation*}
    P
    \begin{pmatrix}
      J_{1} & \\
      & J_{2}
    \end{pmatrix}
    P^{-1}, \qquad P =
    \begin{pmatrix}
      P_{1} & P_{2}
    \end{pmatrix}.
  \end{equation*}
  Here $P_{1}$ is the set of (generalized)-eigenvalues associated with the repeated eigenvalues, and $P_{2}$ are other generalized eigenvalues.  Depending on whether the repeated eigenvalue is real or complex, $J_{1}$ has the following form ($m$ is the multiplicity of $\lambda$)
  \begin{equation}
    \label{eq:rep-eigen-Jordan-blocks}
    J_{1} =
    \begin{cases}
      \lambda I_{m}, & \lambda \in \R, \\
      I_{m}\otimes
      \begin{pmatrix}
        a & -b \\
        b & a
      \end{pmatrix}, & \lambda = a\pm bi \in \C, \; b\ne 0.
    \end{cases}
  \end{equation}

  We claim that $A$ cannot be identifiable for any $\mathbf{x}_{0} \in \R^{d}$. Specifically, the $(A,\mathbf{x}_{0})$-unidentifiable class has the following structure
  \begin{equation}
    \label{eq:rep-eigen-equiv-class}
    \forall B \in [A]_{\mathbf{x}_{0}}, \qquad
    B = P
    \begin{pmatrix}
      J_{1} + UDU' & \\
      & J_{2}
    \end{pmatrix}
    P^{-1}.
  \end{equation}
  Here $D$ is an \emph{arbitrary} matrix in $M_{(m-1)\times (m-1)}$ (for real $\lambda$) or $M_{(2m-2)\times (2m-2)}$ (for complex $\lambda$), $U$ is a semi-orthogonal matrix such that for $\mathbf{v} :=
  \begin{pmatrix}
    I & \mathbf{0}
  \end{pmatrix} P^{-1}\mathbf{x}_{0}$,
  \begin{equation*}
    \begin{gathered}
      \begin{cases}
        U\in M_{m\times (m-1)}, \quad U'\mathbf{v} = \mathbf{0}_{m-1}, & \lambda \in \R \\
        U\in M_{m\times (2m-2)}, \quad
        U'
        \begin{pmatrix}
          \mathbf{v} & \bar{\mathbf{v}}
        \end{pmatrix} = \mathbf{0}, & \lambda = a\pm bi \in \C, \; b\ne 0.
      \end{cases}
    \end{gathered}
  \end{equation*}
\end{theorem}

\begin{remark}
  The assumption that no Jordan block of $A$ contain a \emph{nilpotent} component in Theorem~\ref{thm:repeated-eigen} is to ensure that $A$ can be decomposed in the form of Equation~\eqref{eq:jordan-decomp}, even if there are repeated eigenvalues.  Without this assumption, $A$ may have Jordan blocks that look like
  \begin{equation*}
    J_{1} =
    \begin{pmatrix}
      c_{1} & 1 & 0 \\
      0 & c_{1} & 1 \\
      0 & 0 & c_{1}
    \end{pmatrix}, \qquad J_{2} =
    \begin{pmatrix}
      \begin{matrix}
        a_{2} & -b_{2} \\
        b_{2} & a_{2}
      \end{matrix} &
      \begin{matrix}
        1 & 0 \\
        0 & 1
      \end{matrix} \\
      \begin{matrix}
        0 & 0 \\
        0 & 0
      \end{matrix} & 
      \begin{matrix}
        a_{2} & -b_{2} \\
        b_{2} & a_{2}
      \end{matrix}
    \end{pmatrix}.
  \end{equation*}

We believe it is theoretically possible to prove the equivalence of Lemmas~\ref{thm:matrix-exponential-repeated-real} and \ref{thm:matrix-exponential-repeated-complex}, but it requires much more tedious computations that are not directly related to the main focus of our manuscript, so we decide to focus on those $A$ with two classical types of Jordan blocks in this study.
\end{remark}

\section{Prior Information and Identifiability}
\label{sec:prior-inform-ident}

As shown in Equation~\eqref{eq:example1-equiv-class}, we can write the explicit form of $[A]_{\mathbf{x}_{0}}$ as an affine subspace in $M_{d\times d}$ based on Theorem~\ref{thm:structure}.  This fact not only can serve as a guidance for us to use prior information about $A$ to resolve the identifiability issue, but also implies that we need to check the \textbf{compatibility} of prior information first, because not all prior information on $A$ is compatible with $[A]_{\mathbf{x}_{0}}$. To this end, we will use the following explicit definition of ``prior information'' on $A$ throughout this study.
\begin{definition}[Subset prior information on $A$]
  A piece of subset prior information on $A$ is a subset $\mathcal{S} \subset M_{d\times d}$ to which $A$ belongs. If $\mathcal{S}$ is a linear (affine) subspace of $M_{d\times d}$, we call it linear (affine) prior information on $A$.
\end{definition}

As a remark, subset prior information on $A$ is ``harder'' than a typical Bayesian prior information, in which we are given a prior distribution of $A$ so $A_{ij}$ can still take arbitrary numerical values.

In practice, most examples of subset prior information is affine prior information in which a subset of $x_{ij}$ are assigned known values.  If all these given values equal zero, it is linear prior information; if some of these values are nonzero, it is affine prior information.

\begin{example}\label{example:ex1-3d-prior-info}
  Let us assume that the condition $A_{13}=0$, which is linear prior information, is given to us in Example~\ref{example:ex1-3d}.

  Due to the explicit form of $[A]_{\mathbf{x}_{0}}$ in Equation~\eqref{eq:example1-equiv-class}, we know that $A_{13} = \frac{-3 +b}{4}$, therefore $b=3$, and $\tilde{A}$ defined in Equation~\eqref{eq:example1-Atilde} must be the \emph{unique} system matrix for the solution curve.

  On the other hand, we cannot impose conditions $A_{11}=0$  and $A_{13}=0$ simultaneously, because $A_{11}=0$ implies $\frac{1+b}{4}=0$ or $b=-1$, which contradicts with $b=3$ derived from $A_{13}=0$.
\end{example}

\begin{definition}\label{def:compat-prior-info}
  A piece of subset prior information $\mathcal{S}$ is said to be \textbf{compatible} with $(A,\mathbf{x}_{0})$ if $\mathcal{S} \cap [A]_{\mathbf{x}_{0}}$ is non-empty. It is \textbf{proper} subset prior information for $(A,\mathbf{x}_{0})$ if $\mathcal{S} \cap [A]_{\mathbf{x}_{0}}$ contains only one unique matrix.
\end{definition}

If the prior information is \emph{correct}, that is, the true system matrix $A \in \mathcal{S}$, $\mathcal{S}$ must be compatible with $(A, \mathbf{x}_{0})$ because $A \in [A]_{\mathbf{x}_{0}}$ so $\mathcal{S} \cap [A]_{\mathbf{x}_{0}}$ is non-empty. However, in practice we most often only have an imperfect estimate of $A$ and $\mathbf{x}_{0}$, so the prior information $\mathcal{S}$ \emph{could} be incorrect, thus it is useful to check whether it is compatible with $(\hat{A}, \hat{\mathbf{x}}_{0})$ or not. It is also useful to see if $\mathcal{S}$ is proper for $(A,\mathbf{x}_{0})$.  These questions are answered in part by the next Theorem.

\begin{theorem}
  Let $\mathcal{S}$ be affine prior information represented in this way:
  \begin{equation*}
    S\, \mathrm{vec} \left(A- A_{0}\right) = 0_{L}, \qquad S \in M_{L\times d^{2}}.
  \end{equation*}
  $\mathcal{S}$ is proper prior information if and only if
  \begin{equation}
    \label{eq:proper-lin-prior-info-cond}
    \begin{gathered}
      \mathrm{rank}\big( (\tilde{S} | \mathbf{b}) \big) = \mathrm{rank}(\tilde{S}). \\
      \tilde{S} := S \left( (I_{0}Q^{-1})' \otimes (Q I_{0})\right), \qquad \mathbf{b} := S\, \mathrm{vec}(A_{0}-A).
    \end{gathered}
  \end{equation}
  Here $\mathrm{rank}(\cdot)$ should be understood as the row rank.
\end{theorem}
\begin{proof}
  According to Equation~\eqref{eq:structure-of-equiv-class}, we have
  \begin{equation*}
    \begin{gathered}
      S\, \mathrm{vec} \left(A + Q I_{0} D I_{0}Q^{-1}- A_{0}\right) = 0 \Longrightarrow \tilde{S}\, \mathrm{vec}(D) = \mathbf{b}. \\
    \end{gathered}
  \end{equation*}
  A necessary and sufficient condition for equation $\tilde{S}\, \mathrm{vec}(D) = \mathbf{b}$ have a unique solution is that the row rank of matrix $(\tilde{S} | \mathbf{b})$ equals the row rank of $\mathrm{rank}(\tilde{S})$.
\end{proof}

\section{Two Examples of Two-stage Methods}
\label{sec:more-2stage}

\subsection{Simple Two-stage Method}
\label{sec:simple-2stage}

In this approach, no smoothing is required. We consider $Y$ as a good approximation of $\mathbf{x}(t)$ evaluated at the time grid, and use simple differences in the time direction to estimate $D\mathbf{x}(t)$. The two inner products are estimated by
\begin{equation}
  \label{eq:simple-2stage-est}
  \begin{gathered}
    \hat{\Sigma}_{\mathbf{x}\mathbf{x}} = Y I_{n} Y', \quad K^{\text{simple}} = I_{n}. \\
    \hat{\Sigma}_{D\mathbf{x}, \mathbf{x}} = Y L^{\text{simple}} Y', \\
    L^{\text{simple}} :=
    \frac{1}{\Delta t}
    \begin{pmatrix}
      -1 & & & & \\
      1 & -1  & & & \\
      & 1 & \ddots & & \\
      &  & \ddots & & -1 \\
      & & & & 1
    \end{pmatrix}_{n\times (n-1)}
    \begin{pmatrix}
      I_{n-1} & 0_{n-1}
    \end{pmatrix} \\
    =
    \frac{1}{\Delta t}
    \begin{pmatrix}[ccccc|c]
      -1  & & & & & 0 \\
      1 & -1 & & & & \vdots \\
      & 1 & \ddots & & & \\
      &  & \ddots & & -1 & \vdots \\
      & & & & 1 & 0
    \end{pmatrix}_{n\times n}.
  \end{gathered}
\end{equation}

\subsection{Functional Two-stage Method}
\label{sec:functional-2stage}

\label{ex:functional-2stage}
Let $\bm{\phi}(t) = \left( \phi_{1}(t), \dots, \phi_{B}(t) \right)$ be a basis system defined on $[0,T]$, $J_{\phi} := \uinner{\bm{\phi}(t)}{\bm{\phi}(t)}$ be the inner product matrix of those basis functions so that a candidate solution curve can be represented as $C_{d\times B} \bm{\phi}(t)$, where $C$ is a matrix of linear coefficients. Let $[D]_{B\times B}$ be the matrix representation of the differential operator such that $D\bm{\phi}(t) = [D] \bm{\phi}(t)$. Let $\lambda$ be the roughness penalty parameter, and $H_{\lambda} \in M_{n\times B}$ be the ``hat matrix'' that maps the discrete data to the linear coefficients of the fitted curves, namely, $\hat{\mathbf{x}}(t) = \hat{C}\bm{\phi}(t)$, $\hat{C} = YH_{\lambda}$.

In this case, we have
\begin{equation}
  \label{eq:smooth-2stage-est}
  \begin{gathered}
    \hat{\Sigma}_{\mathbf{x}\mathbf{x}} = \uinner{\hat{\mathbf{x}}(t)}{\hat{\mathbf{x}}(t)} = \hat{C} J_{\phi} \hat{C}' = Y H_{\lambda} J_{\phi} H_{\lambda}' Y', \quad K^{\text{smooth}} = H_{\lambda} J_{\phi} H_{\lambda}'. \\
    \hat{\Sigma}_{D\mathbf{x}, \mathbf{x}} = \uinner{\hat{C} [D]\bm{\phi}(t)}{\hat{C}\bm{\phi}(t)} = Y H_{\lambda} [D] J_{\phi} H_{\lambda}' Y', \quad L^{\text{smooth}} = H_{\lambda} [D] J_{\phi} H_{\lambda}'.
  \end{gathered}
\end{equation}

\section{Additional Examples}
\label{sec:additional-examples}

We provide a few additional examples in this section so that the readers can have a better understanding of the $(A,x_{0})$-identifiability.

\begin{example}
  By definition, a one-dimensional invariant subspace is just the line generated by an eigenvector of $A$.
\end{example}
\begin{example}
  Assume that
  \begin{equation*}
    \begin{gathered}
      A = Q
      \begin{pmatrix}
        J & \\
        & B
      \end{pmatrix} Q^{-1}, \qquad J =
      \begin{pmatrix}
        a & b \\
        -b & a
      \end{pmatrix}, \qquad B \in M_{(d-2)\times (d-2)}. \\
      Q =
      \begin{pmatrix}[c|c|c]
        \mathbf{v}_{1} & \mathbf{v}_{2} & C
      \end{pmatrix}, \quad \mathbf{v}_{1},\; \mathbf{v}_{2} \in \R^{d}, \quad C \in M_{d\times (d-2)}.
    \end{gathered}
  \end{equation*}
  Then $\mathrm{span}(\mathbf{v}_{1}, \mathbf{v}_{2})$ is a 2-dimensional invariant subspace (associated with $J$).  This is because for any $w_{1},w_{2} \in \R$
  \begin{equation*}
    Q
    \begin{pmatrix}
      w_{1} \\
      w_{2} \\
      \mathbf{0}_{d-2}
    \end{pmatrix} =
    \begin{pmatrix}
      \mathbf{v}_{1} & \mathbf{v}_{2} & C
    \end{pmatrix}
    \begin{pmatrix}
      w_{1} \\
      w_{2} \\
      \mathbf{0}_{d-2}
    \end{pmatrix} = w_{1}\mathbf{v}_{1} + w_{2}\mathbf{v}_{2}.
  \end{equation*}
  \begin{equation*}
    \begin{split}
      A (w_{1} \mathbf{v}_{1} + w_{2}\mathbf{v}_{2}) &= Q
      \begin{pmatrix}
        J & \\
        & B
      \end{pmatrix}
      Q^{-1} \cdot  Q
      \begin{pmatrix}
        w_{1} \\
        w_{2} \\ \hline
        \mathbf{0}_{d-2}
      \end{pmatrix} \\
      &= Q
      \begin{pmatrix}
        aw_{1}+bw_{2} \\
        -bw_{1}+aw_{2} \\ \hline
        \mathbf{0}_{d-2}
      \end{pmatrix} \\
      &= (aw_{1}+bw_{2})\mathbf{v}_{1} + (-bw_{1}+aw_{2}) \mathbf{v}_{2} \in \mathrm{span}(\mathbf{v}_{1}, \mathbf{v}_{2}).
    \end{split}
  \end{equation*}
\end{example}

\begin{example}\label{example:I2}
  In this example, we will illustrate that $A$ with a repeated eigenvalue is not identifiable for any $\mathbf{x}_{0} \in \R^{d}$, as declared by Theorem~\ref{thm:repeated-eigen}.

  Let $A = I_{2}$ and $\mathbf{x}_{0}= r(\cos\theta, \sin\theta)'$ for arbitrary $r \in \R^{+}$ and $\theta \in [0, 2\pi)$. The trajectory is
  \begin{equation*}
    \mathbf{x}(t | I_{2}, \mathbf{x}_{0}) =
    re^{t}
    \begin{pmatrix}
      \cos\theta  \\
      \sin\theta
    \end{pmatrix}.
  \end{equation*}

  Now let
  \begin{equation*}
    B =
    \begin{pmatrix}
      1+\sin^{2}\theta & -\sin\theta\cos\theta \\
      -\sin\theta\cos\theta & 1+\cos^{2}\theta
    \end{pmatrix}.
  \end{equation*}
  It is easy to see that
  \begin{equation*}
    \begin{gathered}
      B =
      T_{B}
      \begin{pmatrix}
        1 & 0 \\
        0 & 2
      \end{pmatrix}
      T_{B}', \qquad T_{B} =
      \begin{pmatrix}
        \cos\theta & -\sin\theta \\
        \sin\theta & \cos\theta
      \end{pmatrix}. \\
      \begin{split}
        \mathbf{x}(t|B, \mathbf{x}_{0}) &= e^{tB} \mathbf{x}_{0} = T_{B}
        \begin{pmatrix}
          e^{t} & 0 \\
          0 & e^{2t}
        \end{pmatrix} T_{B}' \mathbf{x}_{0} \\
        &= r \begin{pmatrix}
          \cos\theta & -\sin\theta \\
          \sin\theta & \cos\theta
        \end{pmatrix}
        \begin{pmatrix}
          e^{t} & 0 \\
          0 & e^{2t}
        \end{pmatrix}
        \begin{pmatrix}
          \cos\theta & \sin\theta \\
          -\sin\theta & \cos\theta
        \end{pmatrix}
        \begin{pmatrix}
          \cos\theta \\
          \sin\theta
        \end{pmatrix} \\
        &= r \begin{pmatrix}
          \cos\theta & -\sin\theta \\
          \sin\theta & \cos\theta
        \end{pmatrix}
        \begin{pmatrix}
          e^{t} \\
          0
        \end{pmatrix} =
        re^{t}
        \begin{pmatrix}
          \cos\theta  \\
          \sin\theta
        \end{pmatrix} = \mathbf{x}(t | I_{2}, \mathbf{x}_{0}).
      \end{split}
    \end{gathered}
  \end{equation*}
  Therefore, $A$ and $B$ must be in the same unidentifiable class.
\end{example}

In the next example, we show that certain network topology (the sparsity structure of $A$) always imply unidentifiability.

\begin{example}
  Let $A$ be a matrix with two rows (or columns) with all zeros.  Elementary linear algebra shows that $\lambda=0$ must be an eigenvalue of $A$ with multiplicity greater or equal to two.  As a specific example, consider
  \begin{equation*}
    A =
    \begin{pmatrix}
      3 & 4 & 5 \\
      0 & 0 & 0 \\
      0 & 0 & 0 \\
    \end{pmatrix}, \quad A = T_{A}
    \begin{pmatrix}
      3 & 0 & 0 \\
      0 & 0 & 0 \\
      0 & 0 & 0 \\
    \end{pmatrix} T_{A}^{-1}, \quad T_{A} =
    \begin{pmatrix}
      1 & -\frac{4}{5} & -\frac{5}{\sqrt{34}} \\
      0 & \frac{3}{5} & 0 \\
      0 & 0 & \frac{3}{\sqrt{34}}
    \end{pmatrix}.
  \end{equation*}
  It is easy to see that for an arbitrary initial condition $\mathbf{x}_{0}$ with the following decomposition
  \begin{equation*}
    \mathbf{x}_{0} = T_{A}
    \begin{pmatrix}
      c_{1} \\
      c_{2} \\
      c_{3}
    \end{pmatrix} = c_{1}
    \begin{pmatrix}
      1 \\
      0 \\
      0
    \end{pmatrix} + c_{2}
    \begin{pmatrix}
      -\frac{4}{5} \\
      \frac{3}{5} \\
      0
    \end{pmatrix} + c_{3}
    \begin{pmatrix}
      -\frac{5}{\sqrt{34}} \\
      0 \\
      \frac{3}{\sqrt{34}}
    \end{pmatrix},
  \end{equation*}
  the trajectory is
  \begin{equation*}
    \begin{split}
      \mathbf{x}(t|A,\mathbf{x}_{0}) &= e^{tA} \mathbf{x}_{0} =
      \begin{pmatrix}
        c_{1}e^{3t}  \\
        0 \\
        0 \\
      \end{pmatrix}.
    \end{split}
  \end{equation*}

  Such solution can be generated by the following alternative system
  \begin{equation*}
    B = T_{A}
    \begin{pmatrix}
      3 & 0 & 0 \\
      0 & c_{3}^{2} & -c_{2}c_{3} \\
      0 & -c_{2}c_{3} & c_{2}^{2} \\
    \end{pmatrix} T_{A}^{-1},
  \end{equation*}
  because
  \begin{equation*}
    \begin{split}
      e^{tB} \mathbf{x}_{0} &= T_{A}
      \begin{pmatrix}
        e^{3t} & 0_{2} \\
        0_{2}' & e^{t
          \left(
            \begin{smallmatrix}
              -c_{3} \\
              c_{2}
            \end{smallmatrix}
          \right)
          \left(
            \begin{smallmatrix}
              -c_{3} & c_{2}
            \end{smallmatrix}
          \right)
        } \\
      \end{pmatrix}
      \begin{pmatrix}
        c_{1} \\
        \hline
        c_{2} \\
        c_{3}
      \end{pmatrix} \\
      &= T_{A}
      \begin{pmatrix}
        c_{1}e^{3t}  \\
        0_{2} \\
      \end{pmatrix} = \mathbf{x}(t|A,\mathbf{x}_{0}).
    \end{split}
  \end{equation*}
\end{example}

\begin{example}[An open dense set with arbitrarily small measure]
  \label{exp:open-dense-small}
  Let $\Q_{d\times d} \subset M_{d\times d}$ be the set of matrices with rational entries. It is clear that $\Q_{d\times d}$ is a countable set, so we can enumerate all elements in $\Q_{d\times d}$ as $\left\{ Q_{1}, Q_{2}, \dots \right\}$.  Let
  \begin{equation*}
    \Omega_{a} := \bigcup_{i=1}^{\infty} B(Q_{i}, a\cdot 2^{-i}), \qquad B(q,r) := \left\{ x \in M_{d\times d}:\, \|x-q\|_{F} < r \right\}.
  \end{equation*}
  Apparently, $\Omega_{a}$ is open and dense (because $Q_{d\times d}$ is dense) in $M_{d\times d}$. Furthermore, the Lebesgue of $Q_{d\times d}$ can be made arbitrarily small because
  \begin{equation*}
    \lambda_{d\times d}\left( \Omega_{a} \right) \leqslant \sum_{i=1}^{\infty} \lambda_{d\times d}\left( B(Q_{i}, a\cdot 2^{-i}) \right) = \sum_{i=1}^{\infty} c_{d} a^{d^{2}} 2^{-d^{2}i} < c_{d}a^{d^{2}}.
  \end{equation*}
  Here $c_{d} := \frac{\pi^{d^{2}/2}}{\Gamma(\frac{d^{2}}{2}+1)}$ is a constant that represents the volume of the Frobenius unit ball $B(0_{d\times d}, 1) \subset M_{d\times d}$. Because $a$ is an arbitrary positive number, $\lambda_{d\times d}\left( \Omega_{a} \right)$ can be made arbitrarily small. Due to the absolute continuity between the probability of GinOE and Lebesgue measure, we can also select $a$ so that $P_{\mathrm{GinOE}}\left( \Omega_{a} \right)$ is smaller than any arbitrary positive number.
\end{example}

%%% Local Variables:
%%% mode: latex
%%% TeX-master: "ode_identifiability1_arxiv"
%%% End:

\end{document}